\documentclass[a4paper]{amsart}

\usepackage{amsmath}
\usepackage{amssymb}
\usepackage{ascmac}
\usepackage{amsthm}
\usepackage{color}
\usepackage{enumerate}
\usepackage{mathrsfs}
\usepackage[pdftex]{graphicx}
\usepackage{float}
\usepackage{wrapfig}
\usepackage{layout}
\usepackage{cite}
\usepackage[draft=false,setpagesize=true,pdftex,pdfstartview=FitH,
colorlinks=true,citecolor=blue,pagebackref=true,bookmarks=false]{hyperref}

\newcommand{\R}{\mathbf{R}}
\newcommand{\Z}{\mathbf{Z}}
\newcommand{\C}{\mathbf{C}}
\newcommand{\RN}{\mathbf{R}^N}
\newcommand{\ERN}{\mathbf{R}^{N+1}_+}
\newcommand{\N}{\mathbf{N}}
\newcommand{\e}{\varepsilon}
\newcommand{\rd}{\mathrm{d}}

\newcommand{\scrF}{\mathscr{F}}
\newcommand{\scrL}{\mathscr{L}}
\newcommand{\scrS}{\mathscr{S}}

\newcommand{\hu}{\widehat{u}}

\newcommand{\la}{\left\langle}
\newcommand{\ra}{\right\rangle}

\def\ov#1{\overline{#1}}
\newcommand{\Ha}{H^\alpha(\RN)}
\newcommand{\Har}{H^\alpha_{\rm r}(\RN)}
\newcommand{\Tr}{{\rm Tr\,}}
\def\wh#1{\widehat{#1}}

\theoremstyle{definition}
\newtheorem{definition}{Definition}[section]
\theoremstyle{plain}
\newtheorem{theorem}[definition]{Theorem}

\newtheorem{lemma}[definition]{Lemma}
\newtheorem{proposition}[definition]{Proposition}
\theoremstyle{remark}
\newtheorem{remark}[definition]{Remark}

\usepackage[margin=1.9cm, bottom=3cm, footskip=1.5cm]{geometry}
\usepackage{layout}

\title[]{Existence of solutions of 
scalar field equations with fractional operator}
\author[N. Ikoma]{Norihisa Ikoma}
\date{}
\address{Faculty of Mathematics and Physics, 
Institute of Science and Engineering, 
Kanazawa University, 
Kakuma, Kanazawa, Ishikawa 9201192, JAPAN}
\email{ikoma@se.kanazawa-u.ac.jp}

\subjclass[2010]{
35J60, 35S05}
\keywords{variational method, mountian pass theorem, 
symmetric mountain pass theorem, 
the Pohozaev identity}

\begin{document}
\begin{abstract}
In this paper, 
the existence of least energy solution and infinitely many solutions is proved 
for the equation $(1-\Delta)^\alpha u = f(u)$ in $\RN$ where 
$0<\alpha<1$, $N \geq 2$ and $f(s)$ is a Berestycki--Lions type nonlinearity. 
The characterization of the least energy by the mountain pass value 
is also considered and the existence of optimal path is shown. 
Finally, exploiting these results, 
the existence of positive solution for the equation 
$(1-\Delta)^\alpha u = f(x,u)$ in $\RN$ is established 
under suitable conditions on $f(x,s)$. 
\end{abstract}
\maketitle

\section{Introduction}\label{section:1}
In this paper, we are concerned with the existence of nontrivial solutions of 
	\begin{equation}\label{eq:1}
		\left\{\begin{aligned}
			(1 - \Delta  )^\alpha u &= f(x,u) \quad {\rm in}\ \RN,
			\\
			u & \in H^\alpha(\RN)
		\end{aligned}\right.
	\end{equation}
where $N \geq 2$ and $ 0 < \alpha < 1$. 
The fractional operator $(1-\Delta)^\alpha u$ is defined by 
	\[
		(1-\Delta)^\alpha u := \scrF^{-1}\left( ( 1 + 4\pi^2 |\xi|^2 )^\alpha \wh{u} (\xi )\right), \quad 
		\hat{u}(\xi) := (\mathscr{F}u)(\xi) = \int_{\RN} e^{-2\pi x \cdot \xi} u(x) \rd x
	\]
and $H^\alpha(\RN)$ a fractional Sobolev space consisted by real valued functions, that is, 
	\begin{equation}\label{eq:2}
		H^\alpha(\RN) := \left\{ u \in L^2(\RN, \R) \ \Big| \ 
		\| u \|_{\alpha}^2 := \int_{\RN} (4\pi^2 |\xi|^2 + 1)^\alpha |\hat{u}|^2 \rd \xi  
		< \infty \right\}. 
	\end{equation}
Throughout this paper, we deal with a weak solution of \eqref{eq:1}, namely, 
a function $u \in \Ha$ satisfying 
	\begin{equation}\label{eq:3}
		\int_{\RN} (4\pi^2 |\xi|^2 + 1)^\alpha \wh{u} (\xi) \ov{\wh{\varphi}(\xi)} \rd \xi 
		- \int_{\RN} f(x,u(x)) \varphi (x) \rd x = 0 \quad 
		{\rm for\ all\ } \varphi \in \Ha
	\end{equation}
where  $\ov{a}$ denotes the complex conjugate of $a$. 

The operator $(1-\Delta)^\alpha$ is related to 
the pseudo-relativistic Schr\"odinger operator $(m^2-\Delta)^{1/2} - m$ ($m > 0$) 
and recently a lot of attentions are paid for equations involving them. 
Here we refer to 
\cite{A-15,A-16-1,A-16-2,CS-16,CS-15-1,CS-15-2,CZN-11,CZN-13-1,CZN-13-2,FF-15,FV-15,
MZ-12,Se,TWY-12} and references therein 
for more details and physical context of $(1-\Delta)^\alpha$. 
In these papers, the authors study the existence of nontrivial solution and 
infinitely many solutions for the equations with 
$(m^2 - \Delta)^\alpha$ and various nonlinearities.

	This paper is especially motivated by two papers \cite{FV-15} and \cite{Se}. 
In \cite{FV-15}, the existence of positive solution of \eqref{eq:1} is proved 
under the following conditions on $f(x,s)$: 

\smallskip

(i) $f \in C(\RN \times \R, \R)$.  

(ii) For all $x \in \RN$, $f(x,s) \geq 0$ if $s \geq 0$ and $f(x,s) = 0$ if $s \leq 0$. 

(iii) The function $s \mapsto s^{-1}f(x,s)$ is increasing in $(0,\infty)$ for all $x \in \RN$. 

(iv) There are $1 < p < 2_\alpha^\ast -1 = (N+2\alpha) / (N-2\alpha)$ and $C>0$ such that 
$|f(x,s) | \leq C |s|^{p}$ for all $(x,s) \in \RN \times \R$. 

(v) There exists a $\mu > 2$ such that 
$0 < \mu F(x,s) \leq s f(x,s)$ for all $(x,s) \in \RN \times (0,\infty)$ where 
$F(x,s) := \int_0^s f(x,t) \rd t$. 

(vi) There exist continuous functions $\bar{f}(s)$ and $a(x)$ such that 
$\bar{f}$ satisfies (i)--(v) and $0 \leq f(x,s) - \bar{f}(s) \leq a(x) 
(|s| + |s|^p)$ for all $(x,s) \in \RN \times \R$, $a(x) \to 0$ as $|x| \to \infty$ 
and $\mathcal{L}^N( \{ x \in \RN \ |\  \text{$f(x,s) > \bar{f}(s)$ for all $s>0$}  \} ) > 0$ 
where $\mathcal{L}^N$ denotes the $n$-dimensional Lebesgue measure.

\smallskip

On the other hand, in \cite{Se}, 
the author obtains a nontrivial solution of \eqref{eq:1} with 
$f(x,s) = \lambda b(x) |u|^{p-1} u + c(x) |u|^{q-1} u$ 
under different conditions on $b(x), c(x), p,q$ where $\lambda > 0$ is a constant. 
Among other things, under $1<p,q< 2_\alpha^\ast - 1$ and 
some strict inequality for the mountain pass value (the infimum of the functional 
on the Nehari manifold), 
the existence of positive solution of \eqref{eq:1} is shown. 
However, there is no specific information when the strict inequality holds.

	Our aim of this paper is to observe whether we can handle the more general 
nonlinearity $f(x,s)$ in \eqref{eq:1} and obtain a positive solution. 
First, we treat the case where $f(x,s) = f(s)$ is a Berestycki--Lions type 
nonlinearity, that is, $f(s)$ satisfying (f1)--(f4) below. 
These conditions are introduced in \cite{BL-83-1,BL-83-2} (cf. \cite{BGK-83}) 
for the case $\alpha = 1$ and almost optimal for the existence of 
nontrivial solution. We shall prove the existence of infinitely many solutions and 
least energy solutions with the Pohozaev identity, 
that the least energy coincides with the mountain pass value 
and that there is an optimal path. 
These properties are shown in \cite{JT-03-1,HIT-10} for the case $\alpha = 1$. 
Second, we deal with the case $f(x,s)$ depends on $x$. 
Here, exploiting the optimal path and characterization by the mountain pass value, 
we show the existence of positive solution of \eqref{eq:1}, which 
generalizes the result in \cite{FV-15} and 
enables us to find a simpler sufficient condition than that of \cite{Se} in some case. 
See the comments after Remark \ref{definition:1.4}.

	As stated in the above, we first consider the case $f(x,s) \equiv f(s)$ and 
\eqref{eq:1} becomes 
	\begin{equation}\label{eq:4}
		\left\{\begin{aligned}
			(1 - \Delta  )^\alpha u &= f(u) \quad {\rm in}\ \RN,
			\\
			u & \in H^\alpha(\RN).
		\end{aligned}\right.
	\end{equation}
For \eqref{eq:4}, we assume that 
the nonlinearity $f$ is a Berestycki--Lions type (\cite{BL-83-1,BL-83-2}): 
	\begin{enumerate}
		\item[(f1)] $f \in C(\R,\R)$ and $f(s)$ is odd. 
		\item[(f2)] $ \displaystyle - \infty < \liminf_{s \to 0} 
			\frac{f(s)}{s} \leq \limsup_{s \to 0} \frac{f(s)}{s} < 1$. 
		\item[(f3)] 
			\[
				\lim_{|s| \to \infty} \frac{|f(s)|}{|s|^{2_\alpha^\ast-1}} = 0 
				\quad {\rm where} \ 
				2_\alpha^\ast := \frac{2N}{N-2\alpha}.
			\]
		\item[(f4)] There exists an $s_0>0$ such that 
			\[
				F(s_0) - \frac{1}{2}s_0^2 > 0 \quad 
				{\rm where} \ F(s) := \int_0^s f(t) \rd t.
			\]
	\end{enumerate}
Notice that under (f1)--(f3), \eqref{eq:4} has variational structure, namely, 
a solution of \eqref{eq:4} is characterized as a critical point of 
the following functional (see Lemma \ref{definition:2.1})
	\begin{equation}\label{eq:5}
		I(u) := \frac{1}{2} \| u \|_\alpha^2 - \int_{\RN} F(u) \rd x 
		\in C^1(H^\alpha(\RN), \R).
	\end{equation}
Our first result is the existence of infinitely many solutions of \eqref{eq:4} 
and the characterization of least energy solutions by the mountain pass structure.

	\begin{theorem}\label{definition:1.1}
		Assume $N \geq 2$, $0 < \alpha < 1$ and \emph{(f1)--(f4)}. 
		
		\emph{(i)} There exist infinitely many solutions $(u_n)_{n=1}^\infty$ 
		of \eqref{eq:4} satisfying $I(u_n) \to \infty$ and 
		the Pohozaev identity $P(u_n) = 0$ where 
			\begin{equation}\label{eq:6}
				P(u) := \frac{N-2\alpha}{2} \int_{\RN} \left( 1 + 4 \pi^2 |\xi|^2 \right)^\alpha 
				| \hu |^2 \rd \xi - N \int_{\RN} F(u) \rd x  
				+ \alpha \int_{\RN} \left( 1 + 4\pi^2 |\xi|^2 \right)^{\alpha - 1} 
				| \hu |^2 \rd \xi. 
			\end{equation}
		Moreover, $u_1(x) > 0$ for all $x \in \RN$. 
		
		\emph{(ii)} Assume either $\alpha > 1/2$ or $f(s)$ is locally Lipschitz continuous. 
		Then every solution of \eqref{eq:4} satisfies the Pohozaev identity $P(u)=0$. 
		
		\emph{(iii)} 
		For the following quantities 
						\[
							\begin{aligned}
								c_{\rm MP} &:= \inf_{\gamma \in \Gamma} 
								\max_{0 \leq t \leq 1} I (\gamma(t)), 
								\quad 
								\Gamma := \{ \gamma \in C([0,1], H^\alpha(\RN)) 
								\ |\ \gamma(0) = 0, \ I(\gamma(1)) < 0 \},
								\\
								c_{\rm LES} &:= 
								\inf \left\{ I(u) \ |\ u \not \equiv 0, \ I'(u)=0, \  P(u) = 0 \right\},
								\\
								S_{\rm LES} &:= \left\{ u \in H^\alpha(\RN) \ |\ 
								u \not \equiv 0, \ 
								I'(u) = 0, \ P(u) = 0, \ I(u) = c_{\rm LES} \right\},
							\end{aligned}
						\]
					we have $ S_{\rm LES} \neq \emptyset$ and 
					$c_{\rm MP} = c_{\rm LES} > 0$. 
					Furthermore, for every $v \in S_{\rm LES}$, 
					there exists a $\gamma_v \in \Gamma$ such that 
					$\| \gamma_v(t) \|_{L^\infty} = \| v \|_{L^\infty}$ 
					for $0 < t \leq 1$ and 
						\[
							\max_{0 \leq t \leq 1} I(\gamma_v(1)) = I(v) = c_{\rm LES}. 
						\]
	\end{theorem}

	\begin{remark}\label{definition:1.2}
(i) 
To the author's knowledge, it is not known that 
every weak solution of \eqref{eq:4} satisfies the Pohozaev identity. 
In Proposition \ref{definition:3.6}, we shall show that if a weak solution of \eqref{eq:4} 
is of class $C^1$ with bounded derivatives, then the Pohozaev identity holds.

(ii)
By Theorem \ref{definition:1.1} (ii), when $\alpha > 1/2$ or 
$f(s)$ is locally Lipschitz, we have 
	\[
		\begin{aligned}
			c_{\rm LES} &= \inf \left\{ I(u) \ |\ u \not \equiv 0, \ I'(u) = 0  \right\}, 
			\\
			S_{\rm LES} &= \left\{ u \in H^\alpha(\RN) \ |\ 
			u \not \equiv 0, \ I'(u) = 0, \ I(u) = c_{\rm LES}  \right\}. 
		\end{aligned}
	\]
Thus $c_{\rm LES}$ and $S_{\rm LES}$ coincide with the least energy 
and a set of all least energy solutions in usual sense.

(iii) By the simple scaling, we may deal with 
\[
(m^2 - \Delta)^\alpha u = f(u) \quad {\rm in} \ \RN
\]
where $m>0$.  Indeed, for $m>0$, 
$u(x)$ is a solution of \eqref{eq:4} if and only if $v(x) := u(m^{-1} x)$ satisfies 
\[
(1 - \Delta)^\alpha v = m^{-2\alpha} f(v(x)) \quad {\rm in} \ \RN. 
\]
	\end{remark}

Next, we use Theorem \ref{definition:1.1} to obtain a positive solution 
of \eqref{eq:1}. For $f(x,s)$, assume that 
	\begin{enumerate}
		\item[(F1)] 
		$f(x,s) = - V(x) s + g(x,s)$ where $V \in C(\RN,\R)$, 
		$g \in C(\RN \times \R, \R)$ and $g(x,-s) = - g(x,s)$ 
		for every $(x,s) \in \RN \times \R$.   
		\item[(F2)] 
			\[
				-1 < \inf_{x \in \RN} V(x) \quad 
				{\rm and} \quad 
				\lim_{s \to 0} \sup_{x\in \RN} 
				\left|\frac{g(x,s)}{s}\right| = 0.
			\]
		\item[(F3)] 
			\[
				\lim_{|s| \to \infty} 
				\sup_{x \in \RN} \frac{|g(x,s)|}{|s|^{2_\alpha^\ast-1}} = 0.
			\]
		\item[(F4)] 
			There exist $V_\infty > -1$ and 
			$g_\infty(s) \in C(\R,\R)$ such that 
			as $|x| \to \infty$, 
			$V(x) \to V_\infty$ and $g(x,s) \to g_\infty(s)$ in $L^\infty_{\rm loc}(\RN)$ 
			where $g_\infty(s)$ is locally Lipschitz continuous provided $0<\alpha \leq 1/2$. 
			Moreover, $0 \leq F(x,s) - F_\infty(s)$ holds 
			for all $x \in \RN$ and $s \in \R$ where 
			$F(x,s) := \int_0^s f(x,t) \rd t$, 
			$f_\infty(s) := -V_\infty s + g_\infty(s)$ and 
			$F_\infty(s) := \int_0^s f_\infty(t) \rd t$. 
		\item[(F5)] 
			There exist $\mu > 2$ and $s_1>0$ such that 
				\[
					0< \mu G(x,s) \leq g(x,s) s \quad {\rm for\ each} \ (x,s) \in \RN 
					\times \R \setminus \{0\}, 
					\quad 
					\inf_{x \in \RN} G(x,s_1) > 0
				\]
			where $G(x,s) := \int_0^s g(x,t) \rd t$. 
	\end{enumerate}

Under these conditions, we have

	\begin{theorem}\label{definition:1.3}
		Assume \emph{(F1)--(F5)}. 
		Then \eqref{eq:1} admits a positive solution. 
	\end{theorem}
	\begin{remark}\label{definition:1.4}
(i) In (F4), when $0<\alpha \leq 1/2$, we assume that 
$g_\infty(s)$ is local Lipschitz in $s$, however, not for $g(x,s)$. 

(ii) (F5) is mainly used to find a bounded Palais--Smale 
sequence. If we assume the existence of bounded Palais--Smale sequence 
at the mountain pass level, we can show the existence of nontrivial solution 
of \eqref{eq:1} in the more general setting. See Proposition \ref{definition:4.3}. 

(iii) Another way to obtain bounded Palais--Smale sequences 
is to exploit the Pohozaev identity. When $\alpha = 1$, for example, we refer to \cite{AP-09,JT-05}. 
When $0<\alpha < 1$, we also have the Pohozaev identity (see \eqref{eq:46}) 
and it might be useful to get a bounded Palais--Smale sequence in the case $0<\alpha < 1$. 
	\end{remark}

	Now we compare our result with the previous results. 
We first consider \eqref{eq:4}. 
The most related results are \cite{A-16-2,CS-16,FV-15,TWY-12}. 
In these papers, the authors study \eqref{eq:4} with $f(s) = |s|^{p-1}s$ or 
$f(s) = (1-\mu ) s  + |s|^{p-1} s$ where 
$1 < p < 2_\alpha^\ast - 1$ and $\mu>0$, and show the existence of least energy solution 
(or ground state solution) and infinitely many solutions. 
Clearly, Theorem \ref{definition:1.1} improves these results. 
Furthermore, 
in \cite{A-16-2,Se}, the authors raise a question that 
one can prove the existence of least energy solution and 
infinitely many solutions of \eqref{eq:4} with general nonlinearity. 
Theorem \ref{definition:1.1} answers this question. 
For the fractional Laplacian $(-\Delta)^\alpha$ with general nonlinearity, 
we refer to the work \cite{CW-13}.

	On the other hand, for \eqref{eq:1}, the existence of positive solution 
is proved in \cite{FV-15,Se}. 
It is easily checked that 
Theorem \ref{definition:1.3} is a generalization of the result in \cite{FV-15}. 
In addition, suppose that $f(x,s) = \lambda b(x) |u|^{p-1} u + c(x) |u|^{q-1} u$, 
$b(x) \geq \underline{b} = \lim_{|x| \to \infty} b(x) > 0$ and 
$ c (x) \geq 0 = \lim_{|x| \to \infty} c(x)$. 
Then we can apply Theorem \ref{definition:1.3} 
to get a positive solution of \eqref{eq:1} for every $\lambda >0$ and 
$1 < p,q < 2_\alpha^\ast - 1$. 
Hence, in this case, we find the simpler sufficient condition than 
that of \cite{Se} for the existence of nontrivial solution

Finally, we comment on the proofs of Theorems \ref{definition:1.1} and 
\ref{definition:1.3}. Our arguments are variational and 
we find critical points of $I$ defined by \eqref{eq:5} and 
	\[
		J(u) := \frac{1}{2} \| u \|_\alpha^2 
		- \int_{\RN} F(x,u(x)) \rd x \in C^1(\Ha, \R). 
	\]
To show the existence of critical points of $I$, 
we use the arguments in \cite{J-97,HIT-10} and 
introduce the augmented functional based on the scaling:
	\[
		\tilde{I}(\theta, u) := I( u ( \cdot/ e^{\theta}  ) ) \in C^1
		( \R \times \Ha , \R ). 
	\]
As already pointed out in \cite{A-16-2,Se}, 
$-\Delta$ and $(-\Delta)^\alpha$ are homogenous in scaling, 
however, $(1-\Delta)^\alpha$ is not. 
Nevertheless, $\tilde{I}$ still helps us to find bounded 
Palais--Smale sequences.

	Next, we turn to Theorem \ref{definition:1.3}. 
We use the idea of the concentration compactness lemma 
(\cite{L-84-1,L-84-2,JT-04}) and compare the mountain pass values. 
First, we treat the general setting and 
exploiting Theorem \ref{definition:1.1}, 
we prove that it suffices to find a bounded Palais--Smale sequence of $J$ 
at the mountain pass level. To this end, we observe the behavior 
of any bounded Palais--Smale sequence of $J$. 
After that, we shall prove Theorem \ref{definition:1.3}.

	This paper is organized as follows. 
In sections \ref{section:2} and \ref{section:3}, we introduce 
the augmented functional $\tilde{I}(\theta, u)$, prove its properties and 
show Theorem \ref{definition:1.1}. 
Section \ref{section:4} is devoted to proving Theorem \ref{definition:1.3}. 
In Appendix, we collect some technical lemmas and prove 
a Br\'ezis-Kato type result (Proposition \ref{definition:3.5}).


\section{Variatonal setting}
\label{section:2}


To prove Theorems \ref{definition:1.1} and \ref{definition:1.3}, 
we employ the variational methods. We first consider \eqref{eq:4} and 
prove Theorem \ref{definition:1.1}. 
In what follows, we always assume (f1)--(f4) and use the following notations: 
For $K= \R, \C$, $\scrS(\RN, K)$ denotes the Schwartz class consisting of 
$K$-valued functions. Moreover, set 
$\Har := \{ u \in \Ha \ |\ \text{$u$ is radial} \}$. 
Recalling \eqref{eq:2}, we begin with the following lemma. 

	\begin{lemma}\label{definition:2.1}
\emph{(i)} The space $H^\alpha(\RN)$ is a Hilbert space over $\R$ 
under the following scalar product: 
	\[
		\la u , v \ra_\alpha := 
		\int_{\RN} (4\pi^2 |\xi|^2 + 1)^\alpha \hat{u}(\xi) 
		\ov{\hat{v}(\xi)} \rd \xi .
	\]
Notice that $\| u \|_\alpha^2 = \la u , u \ra_\alpha$. 

\emph{(ii) (\cite{L-82})} 
The embedding $\Har \subset L^p(\RN)$ is compact for $2 < p < 2_\alpha^\ast$. 

\emph{(iii)} The functional $I$ in \eqref{eq:5} belongs to $C^1(\Ha,\R)$ and 
	\begin{equation}\label{eq:7}
		I'(u) \varphi = \int_{\RN} (1 + 4\pi^2 |\xi|^2)^\alpha 
		\hat{u} \ov{\hat{\varphi}} \rd \xi - \int_{\RN} f(u) \varphi \rd x 
		\quad {\rm for\ all}\ \varphi \in H^\alpha(\RN).
	\end{equation}
In particular, if $I'(u) = 0$, then $u$ satisfies \eqref{eq:4} 
in $(\scrS(\RN,\C))^\ast$, that is, for all $\varphi \in  \scrS(\RN, \C)$, 
	\begin{equation}\label{eq:8}
		\la (1-\Delta)^\alpha u , \varphi \ra 
		= \int_{\RN} (1 + 4\pi^2 |\xi|^2)^\alpha \hat{u} (\xi) (\scrF^{-1} \varphi) (\xi) 
		\rd \xi = \int_{\RN} f(u) \varphi \rd x.  
	\end{equation}
The same holds true for $I|_{\Har}$. 
	\end{lemma}

	\begin{proof}
(i) We only check $\la u, v \ra_\alpha \in \R$ for any $u,v \in \Ha$. 
Put  
	\[
		G_{2\alpha}(x) := \frac{1}{(4\pi)^\alpha} \frac{1}{\Gamma(\alpha)} 
		\int_0^\infty e^{-\pi |x|^2 / t} e^{-t/4\pi} t^{(2\alpha - N)/2} \frac{\rd t}{t}. 
	\]
Then it is known that (see \cite[Chapter V]{St})
	\begin{equation}\label{eq:9}
		\begin{aligned}
			&\widehat{G_{2\alpha}}(\xi) = (4\pi^2 |\xi|^2 + 1)^{-\alpha}, 
			\quad \| G_{2\alpha} \|_{L^1} = 1, \\
			& 0 \leq G_{2\alpha}(x) \leq 
			C_0 \left( |x|^{N-2\alpha} \chi_{B_1(0)}(x) + e^{-c_1 |x|} \chi_{(B_1(0))^c}(x) \right)
		\end{aligned}
	\end{equation}
for some $c_1>0$ where $B_1(0) :=\{ x\in \RN\ | \ |x| < 1 \}$ and 
$\chi_{A}$ is a characteristic function of $A$. 
Moreover, for every $\varphi \in \scrS(\RN,\R)$, the equation 
	\[
		(-\Delta + 1)^\alpha u = \varphi \quad {\rm in} \ \RN, \quad 
		u \in H^\alpha(\RN)
	\]
has a unique solution $u$ expressed as $u = G_{2\alpha} \ast \varphi \in \scrS(\RN,\R)$ 
due to \eqref{eq:9}. 
For this $u$, if $v \in \scrS(\RN,\R)$, then 
	\[
		\begin{aligned}
			\la u , v \ra_\alpha &= \int_{\RN} (1+4\pi|\xi|^2)^\alpha \hat{u}(\xi) 
			\ov{\hat{v}(\xi)} \rd \xi 
			= \int_{\RN} (1+4\pi |\xi|^2)^\alpha \widehat{G_{2\alpha} \ast \varphi} 
			\ov{\hat{v}(\xi)} \rd \xi 
			\\
			&= \int_{\RN} \hat{\varphi} \ov{\hat{v}} \rd \xi = \int_{\RN} \varphi v \rd x 
			\in \R.
		\end{aligned}
	\]
Here at the last equality, we used the Plancherel theorem. 
Therefore, if $\varphi, v \in \scrS(\RN,\R)$ and $u = G_{2\alpha} \ast \varphi$, 
then $\la u , v \ra_\alpha \in \R$. 
Since $\scrS (\RN,\R)$ is dense in $H^\alpha(\RN)$, we have 
	\[
		\la G_{2\alpha} \ast \varphi , v \ra_\alpha \in \R \quad 
		{\rm for\ all} \ \varphi \in \scrS(\RN,\R), \ v \in H^\alpha(\RN).
	\]
Finally, the map $\varphi \mapsto G_{2\alpha} \ast \varphi : 
\scrS(\RN,\R) \to \scrS(\RN,\R)$ is bijective, by the density argument, 
we obtain $\la u , v \ra_{\alpha} \in \R$ for every $u,v \in H^\alpha(\RN)$.

(ii) This is proved in \cite{L-82}.

(iii) Noting $\la u,v \ra_\alpha \in \R$ for all $u,v \in \Ha$ and 
(f1)--(f4), it is easy to check $I \in C^1(H^\alpha(\RN), \R)$ and \eqref{eq:7}. 
For \eqref{eq:8}, we see from \eqref{eq:7} that 
	\[
			\int_{\RN} (1 + 4\pi^2 |\xi|^2)^\alpha \hat{u} \ov{\hat{\psi}}  \rd \xi 
			= \int_{\RN} f(u) \ov{\psi} \rd x \quad 
			{\rm for \ all} \ \psi \in \scrS(\RN, \mathbf{C}). 
	\]
Then setting $\varphi(x) := \ov{\psi(x)}$ and noting 
$ \ov{\scrF \psi} = \scrF^{-1} \ov{\psi} = \scrF^{-1} \varphi $, 
one observes that \eqref{eq:8} holds.

The last assertion follows from the principle of symmetric criticality. 
See \cite{W}. 
\end{proof}

Hereafter, we shall look for critical points of $I|_{\Har}$. 
Following the arguments in \cite{HIT-10}, 
we first introduce a comparison functional $\bar{I}(u)$, 
which plays a role to show that the minimax values $c_n$ defined in \eqref{eq:18} 
diverge as $n \to \infty$. 
To this end, we modify the nonlinearity $f(s)$. 
By (f2), choose $\delta_0>0$ and $s_1>0$ such that 
	\begin{equation}\label{eq:10}
		s f(s) \leq (1-2 \delta_0) s^2 \quad {\rm for\ all} \ |s| \leq s_1.
	\end{equation}
Fixing a $p_0 \in (1,2_\alpha^\ast - 1)$, set
	\[
		h(s) := \left\{ \begin{aligned}
			& \left(f(s) - (1-\delta_0) s\right)_+ , & &{\rm if} \ s \geq 0,\\
			& - h(-s) & &{\rm if} \ s < 0, 
		\end{aligned} \right. 
		\qquad 
		\bar{h}(s) := \left\{ \begin{aligned}
			& s^{p_0} \sup_{0 < t < s} \frac{h(t)}{t^{p_0}} 
			& &{\rm if} \ s > 0, 
			\\
			& 0 & &{\rm if} \ s = 0,\\
			& -\bar{h}(-s) & &{\rm if} \ s<0
		\end{aligned} \right.
	\]
where $a_+ := \max \{0,a\}$.
Finally, put $\bar{H}(s) := \int_0^s \bar{h}(t) \rd t$. Then 
	\begin{lemma}\label{definition:2.2}
		\emph{(i)} 
			$\bar{h} \in C(\R)$ is odd, $\bar{h}(s) \geq 0$ for $s \geq 0$, 
			$\bar{h} \not \equiv 0$ and $\bar{h}$ satisfies \emph{(f3)}. 
			
		\emph{(ii)} 
			There exists an $s_2>0$ such that 
			$\bar{h}(s) = 0 = \bar{H}(s)$ for all $|s| \leq s_2$. 
			In particular, there is a $C_0>0$ such that 
			$|\bar{h}(s)s| + |\bar{H}(s)| \leq C_0 |s|^{2_\alpha^\ast}$ 
			for each $s \in \R$. 
		
		\emph{(iii)}
			$ 0 \leq (p_0+1) \bar{H}(s) \leq s \bar{h}(s)$ for any $s \in \R$. 
		
		\emph{(iv)} 
			$F(s) - (1-\delta_0) s^2/ 2 \leq \bar{H}(s)$ for every $s \in \R$. 
		
		\emph{(v)} 
			Let $(u_n) \subset \Har$ satisfy $u_n \rightharpoonup u_0$ 
			weakly in $\Har$ and $u_n(x) \to u_0(x)$ for a.e. $x \in \RN$. Then 
				\[
					\bar{H}(u_n)  \to \bar{H}(u_0)  \quad 
					{\rm strongly \ in}\ L^1(\RN), \quad 
					\bar{h}(u_n) \to \bar{h}(u_0) 
					\quad {\rm strongly\ in} \ L^{2N/(N+2\alpha)} (\RN). 
				\]
	\end{lemma}

	\begin{proof}
Since one can prove (i)--(iv) in a similar way to 
\cite[Lemma 2.1 and Corollary 2.2]{HIT-10}, 
we omit the details. Now we shall prove (v). 
Since both of the assertions can be proved in a similar way, 
we only treat $\bar{h}(u_n) \to \bar{h}(u_0)$ strongly in $L^{2N/(N+2\alpha)}(\RN)$. 
Noting that $\bar{h}$ satisfies (f3) thanks to the assertion (i), for each $\e>0$ 
there exists an $s_\e>0$ such that 
	\begin{equation}\label{eq:11}
		|\bar{h}(s)|^{2N/(N+2\alpha)} \leq \e |s|^{2_\alpha^\ast} \quad 
		{\rm if} \ |s| \geq s_\e. 
	\end{equation}

		Set $[|u_n| < a] := \{ x \in \RN \ |\ |u_n(x)| < a \}$, 
$\chi_{n,\e}(x) := \chi_{ [|u_n| < s_\e] }(x)$ 
and $\chi_{0,\e}(x) := \chi_{ [|u_0| < s_\e] } (x) $. 
Using $\bar{h}(u_n(x)) = \bar{h} ( \chi_{n,\e}(x) u_n(x) ) + \bar{h} ( (1-\chi_{n,\e}(x) ) u_n(x) )$ 
and writing $v_n(x) := \chi_{n,\e}(x) u_n(x) $, 
$v_0(x) := \chi_{0,\e}(x) u_0(x)$,  $w_n(x) := (1 - \chi_{n,\e}(x) ) u_n(x)$ 
and $w_0(x) := (1 - \chi_{0,\e}(x) ) u_0(x)$, 
we have  
	\begin{equation}\label{eq:12}
		| \bar{h}(u_n) - \bar{h}(u_0)| 
		\leq | \bar{h}( \chi_{0,\e} v_n) - \bar{h}(v_0)| 
		+ |\bar{h} ((1-\chi_{0,\e}) v_n ) | 
		+ |\bar{h}(w_n) - \bar{h}(w_0)|.
	\end{equation}
Since $w_n(x) \neq 0$ implies $|w_n(x)| \geq s_\e$, 
it follows from \eqref{eq:11} that 
	\begin{equation}\label{eq:13}
		\begin{aligned}
		\sup_{n \geq 1} 
		\int_{\RN} |\bar{h}(w_n) - \bar{h}(w_0)|^{2N/(N+2\alpha)} \rd x 
		& \leq 
		C_0 \sup_{n \geq 1} \int_{\RN} \left( |\bar{h}(w_n)|^{2N/(N+2\alpha)} 
		+ |\bar{h}(w_0)|^{2N/(N+2\alpha)} \right) \rd x 
		\\
		&\leq C \e \sup_{n \geq 1} 
		\left( \| w_n \|_{L^{2_\alpha^\ast}}^{2^\ast_\alpha}  
		+ \| w_0 \|_{L^{2_\alpha^\ast}}^{2^\ast_\alpha} \right) 
		\leq C \e. 
		\end{aligned}
	\end{equation}

	Recalling $u_n(x) \to u_0(x)$ and the definition of 
$\chi_{0,\e}(x)$, we observe that 
	\[
		\limsup_{n \to \infty } |(1-\chi_{0,\e})(x) v_n(x)| 
		\leq \chi_{[ |u_0| = s_\e ]}(x) |s_\e| 
		\quad \text{for a.e. $x \in \RN$}.  
	\]
Hence, using \eqref{eq:11} and 
$\chi_{[ |u_0| = s_\e ]}(x) \leq (1-\chi_{0,\e}(x))$, we obtain 
	\begin{equation}\label{eq:14}
		\begin{aligned}
			\limsup_{n\to \infty} \int_{\RN} 
			| \bar{h} ( (1-\chi_{0,\e})(x) v_n(x)  ) |^{2N/(N+2\alpha)} \rd x 
			&\leq \int_{\RN} |\bar{h}(\chi_{[ |u_0| = s_\e ]}(x) |s_\e|)|^{2N/(N+2\alpha)} \rd x 
			\\
			& \leq \int_{\RN} |\bar{h} (w_0) |^{2N/(N+2\alpha)} \rd x 
			\leq C \e. 
		\end{aligned}
	\end{equation}

		On the other hand, since $\chi_{0,\e}(x) v_n(x) \to v_0(x)$ for a.e. $x \in \RN$, 
noting that  $u_n \to u_0$ strongly in $L^p(\RN)$ for $2<p< 2_\alpha^\ast$ 
due to Lemma \ref{definition:2.1} (ii) and that $|v_n(x)|, |v_0(x)| \leq s_\e$, 
we have $\chi_{0,\e} v_n \to v_0$ strongly in $L^p(\RN)$ for $2 < p < \infty$. 
Thus, by the assertion (ii), it is easily seen that 
	\begin{equation}\label{eq:15}
		\lim_{n\to \infty} 
		\int_{\RN} | \bar{h}(\chi_{0,\e} v_n) - \bar{h}(v_0)|^{2N/(N+2\alpha)} \rd x 
		= 0.
	\end{equation}
Collecting \eqref{eq:12}--\eqref{eq:15}, one sees
	\[
		\limsup_{n\to \infty} \| \bar{h}(u_n) - \bar{h}(u_0) 
		\|_{L^{2N/(N+2\alpha)}}^{2N/(N+2\alpha)}
		\leq C \e. 
	\]
Since $\e>0$ is arbitrary, $\bar{h}(u_n) \to \bar{h}(u_0)$ strongly 
in $L^{2N/(N+2\alpha)}(\RN)$. 
	\end{proof}

	Next, from
	\[
		\begin{aligned}
			I(u) &= \frac{1}{2} \| u \|_\alpha^2 
				- \frac{1-\delta_0}{2} \int_{\RN} u^2 \rd x 
				- \int_{\RN} F(u) - \frac{1-\delta_0}{2} u^2 \rd x
				\\
				&\geq \frac{\delta_0}{2} \| u \|_\alpha^2 
				- \int_{\RN} F(u) - \frac{1-\delta_0}{2} u^2 \rd x,
		\end{aligned}
	\]
we define a comparison functional $\bar{I}(u)$ by 
	\[
		\bar{I}(u) := \frac{\delta_0}{2} \| u \|_\alpha^2 
		- \int_{\RN} \bar{H}(u) \rd x. 
	\]

	\begin{lemma}\label{definition:2.3}
		\emph{(i)} The inequality $\bar{I}(u) \leq I(u)$ holds for any $u \in \Ha$.  
		Moreover, there exists a $\rho_0>0$ such that 
			\[
				0 < \inf_{\|u \|_{\alpha} = \rho_0 } \bar{I}(u), \quad 
				\bar{I}(u) \geq 0 \quad {\rm if}\ \| u \|_\alpha \leq \rho_0.
			\]
		
		\emph{(ii)}
		The functional $\bar{I}$ satisfies the Palais--Smale condition. 
		
		\emph{(iii)} 
		For each $n \geq 1$, 
		there exists a $\gamma_n \in C(\partial D_n, \Har)$ such that 
			\[
				\gamma_n(-\sigma) = - \gamma_n (\sigma), \quad 
				I(\gamma_n(\sigma)) < 0 \quad {\rm for\ each} \ \sigma \in \partial D_n
			\]
		where $D_n := \{ \sigma = (\sigma_1,\ldots, \sigma_n) \in \R^n \ |\ 
		|\sigma | \leq 1 \}$. 
	\end{lemma}
	\begin{proof}
(i) The inequality $\bar{I}(u) \leq I(u)$ is clear from the definition and 
Lemma \ref{definition:2.2}. 
Next by Lemma \ref{definition:2.2} (ii), we have 
$|\bar{H}(s)| \leq C|s|^{2_\alpha^\ast}$ for all $s \in \R$. 
Thus it follows from Sobolev's inequality that 
	\[
		\begin{aligned}
			\bar{I}(u) & \geq \frac{\delta_0}{2} \| u \|_{\alpha}^2 
			- C \int_{\RN} |u|^{2_\alpha^\ast} \rd x 
			\geq \frac{\delta_0}{2} \| u \|_{\alpha}^2 
			- C \| u \|_{\alpha}^{2^\ast_\alpha}. 
		\end{aligned}
	\]
Noting $2<2_\alpha^\ast$ and choosing $\rho_0>0$ sufficiently small, we get 
	\[
		\inf_{\| u \|_\alpha = \rho_0} I(u) > 0, \quad 
		I(u) \geq 0 \quad {\rm if}\ \| u \|_\alpha \leq \rho_0.
	\]

(ii) 
Since the nonlinearity $\bar{h}$ satisfies the global Ambrosetti--Rabinowitz 
condition (Lemma \ref{definition:2.2} (iii)), 
following the argument in \cite{R} (cf. proof of Theorem \ref{definition:1.3} below) 
and using Lemma \ref{definition:2.2} (v), 
we can prove that $\bar{I}$ satisfies the Palais--Smale condition and 
we omit the details. 

(iii) Since $f(s) - s$ satisfies the Berestycki--Lions type conditions 
(see \cite{BL-83-1,BL-83-2}), 
as in \cite[Theorem 10]{BL-83-2}, we may find a  map 
$\pi_n \in C( \partial D_n , H^1_{\rm r}(\RN))$ with the properties 
	\[
		0 \not \in \pi_n(\partial D_n), \quad \pi_n(-\sigma) = - \pi_n(\sigma),
		\quad 
		\int_{\RN} F(\pi_n(\sigma)) - \frac{1}{2} (\pi_n(\sigma))^2 \rd x 
		\geq 1 \quad {\rm for\ all} \ \sigma \in \partial D_n. 
	\]
Set $\gamma_n(\sigma)(x) := \pi_n(\sigma) (x/t)$ for $t>0$. 
Then for sufficiently large $t>0$, 
it follows from $\wh{\gamma_n(\sigma)} (\xi) = t^N \wh{\pi_n(\sigma)}(t \xi)$ and 
the inequality $(1+s)^\alpha \leq 1 + s^\alpha$ for $s \geq 0$ that 
	\[
		\begin{aligned}
			I(\gamma_n(\sigma)) &= 
			\frac{t^N}{2} 
			\int_{\RN} \left( 1 + \frac{4\pi^2 |\xi|^2}{t^2} \right)^\alpha 
			| \wh{\pi_n(\sigma)} |^2 \rd \xi 
			- t^N \int_{\RN} F( \pi_n(\sigma) ) \rd x
			\\
			& \leq \frac{t^{N-2\alpha}}{2} \int_{\RN} (4\pi^2|\xi|^2)^\alpha 
			| \widehat{\pi_n(\sigma)} |^2 \rd \xi 
			- t^N \int_{\RN} F(\pi_n(\sigma)) 
			- \frac{1}{2} \left( \pi_n(\sigma) \right)^2 \rd x 
			\\
			& \leq \frac{t^{N-2\alpha}}{2} \int_{\RN} (4\pi^2|\xi|^2)^\alpha 
						| \widehat{\pi_n(\sigma)} |^2 \rd \xi - t^N < 0 
						\quad {\rm for\ all} \ \sigma \in \partial D_n.
		\end{aligned}
	\]
Since $H^1_{\rm r}(\RN) \subset \Har$, we have 
$\gamma_n \in C(\partial D_n, \Har)$ and 
complete the proof. 
	\end{proof}

	\begin{remark}\label{definition:2.4} 
When $n = 1$, we can assume that 
$\gamma_1(1)(x) \geq 0$ for each $x \in \RN$, 
$\gamma_1(1)(|x|) = \gamma_1(1)(x)$ and
$r \mapsto \gamma_1(1)(r)$ is piecewise linear and nonincreasing. 
See \cite{BL-83-1,BL-83-2} (cf. the proof of Proposition \ref{definition:4.1} below). 
	\end{remark}

		Now we introduce an auxiliary functional 
based on the scaling property as in \cite{HIT-10,J-97}. 
For this purpose, we set 
	\[
		u_\theta (x) := u( e^{-\theta} x ). 
	\]
Then we have 
	\begin{equation}\label{eq:16}
		\widehat{u_\theta}(\xi) = e^{N \theta} \hat{u} (e^{\theta} \xi), 
		\quad 
		I(u_\theta) = \frac{e^{N\theta}}{2} \int_{\RN} \left( 1 + 4\pi^2 \frac{|\xi|^2}{e^{2\theta}} 
				\right)^\alpha |\hat{u}|^2 \rd \xi - e^{N\theta} \int_{\RN} F(u) \rd x. 
	\end{equation}
From this, we define $\tilde{I}(\theta, u)$ by 
	\[
		\tilde{I} (\theta, u) := 
		\frac{e^{N\theta}}{2} \int_{\RN} \left( 1 + 4\pi^2 \frac{|\xi|^2}{e^{2\theta}} 
		\right)^\alpha |\hat{u}|^2 \rd \xi - e^{N\theta} \int_{\RN} F(u) \rd x. 
	\]
It is easily seen that 
$\tilde{I} \in C^1(\R \times \Har , \R)$ and 
$u$ is a critical point of $I$ if $(0,u)$ is a critical point of $\tilde{I}$. 
Furthermore, we see the following relation between $\tilde{I}$ and 
$P$ (see \eqref{eq:6} for the definition of $P(u)$): 
for all $u \in \Ha$ and $\theta \in \R$, 
	\begin{equation}\label{eq:17}
		\begin{aligned}
			D_\theta \tilde{I} (\theta, u) 
			&= \frac{N}{2} e^{N\theta} \int_{\RN} 
			(1 + 4 \pi^2 e^{-2\theta} |\xi|^2 )^\alpha |\hat{u}(\xi)|^2 \rd \xi 
			\\
			& \qquad 
			-  \alpha e^{N\theta} \int_{\RN}
			(1 + 4 \pi^2 e^{-2\theta}  |\xi|^2)^{\alpha-1} e^{-2\theta} 
			4\pi^2 |\xi|^2 |\hat{u}(\xi)|^2 \rd \xi 
			- N e^{N\theta} \int_{\RN} F(u) \rd x
			\\
			&= \frac{N}{2} \int_{\RN} 
						(1 + 4 \pi^2 |\xi|^2 )^\alpha |\widehat{u_\theta}(\xi)|^2 \rd \xi 
			\\
			& \qquad 
				-  \alpha \int_{\RN} (1 + 4 \pi^2 |\xi|^2)^{\alpha-1} 4\pi^2 |\xi|^2 
				 |\widehat{u_\theta}(\xi)|^2 \rd \xi - N \int_{\RN} F(u) \rd x
			\\
			&= D_\theta \tilde{I} (0, u_\theta) 
			= \frac{N-2\alpha}{2} \| u_\theta \|_\alpha^2 
			+ \alpha \int_{\RN} (1+ 4\pi^2 |\xi|^2)^{\alpha -1} |\wh{u_\theta}|^2 \rd \xi 
			- N \int_{\RN} F(u_\theta ) \rd x 
			\\
			&= 
			 P(u_\theta).
		\end{aligned}
	\end{equation}
This functional is useful to generate a bounded Palais--Smale sequence $(u_k)$ 
with $P(u_k) \to 0$.

Recalling Lemma \ref{definition:2.3}, for every $n \geq 1$, we define 
	\begin{equation}\label{eq:18}
		\begin{aligned}
		c_n &:= \inf_{\gamma \in \Gamma_n} \max_{\sigma \in D_n} 
		I (\gamma(\sigma)), 
		\quad 
		\tilde{c}_n := 
		 \inf_{\tilde{\gamma} \in \tilde{\Gamma}_n} 
		 \max_{\sigma \in D_n} 
		\tilde{I} (\tilde{\gamma}(\sigma)), 
		\quad 
		d_n  := \inf_{\gamma \in \Gamma_n} \max_{\sigma \in D_n} 
				\bar{I} (\gamma(\sigma)), 
				\\
		\Gamma_n &:= \left\{ \gamma \in C(D_n,\Har) \ |\ 
				\gamma(-\sigma) = -\gamma(\sigma), \ 
				\gamma = \gamma_n \ {\rm on}\ \partial D_n \right\}, 
				\\
		\tilde{\Gamma}_n &:= 
		\{ \tilde{\gamma}(\sigma) 
		=(\theta(\sigma), \gamma(\sigma) ) \in C(D_n, \R  \times \Har ) \ |\ 
		\\
		& \hspace{3cm}
		\theta(-\sigma) = \theta(\sigma) \ {\rm for\ all}\ \sigma \in D_n,\ 
		\theta(\sigma) = 0 \ {\rm on}\ \partial D_n, \ 
		\gamma \in \Gamma_n
		 \}
		\end{aligned}
	\end{equation}
Remark that $\Gamma_n \neq \emptyset$ since $\gamma_{n,0} \in \Gamma_n$ 
where $\gamma_{n,0}(\sigma) := |\sigma| \gamma_n(\sigma/|\sigma|) $ 
when $\sigma \in D_n \setminus \{0\}$ and 
$\gamma_{n,0}(0) := 0$. Furthermore, we have 
	\begin{lemma}\label{definition:2.5}
		For all $n \in \N$, $d_n \leq c_n = \tilde{c}_n $ hold and 
		$ d_n \to \infty$ as $n \to \infty$.  
	\end{lemma}

	\begin{proof}
By definition and Lemma \ref{definition:2.3}, $d_n \leq c_n$ is clear. 
Moreover, noting $(0,\gamma) \in \tilde{\Gamma}_n$ for all $\gamma \in \Gamma_n$ 
and $I(u) = \tilde{I}(0,u)$, we have $\tilde{c}_n \leq c_n$ for all $n \in \N$. 
On the other hand, let $\tilde{\gamma} =(\theta, \gamma) \in \tilde{\Gamma}_n$ 
and put $\zeta(\sigma) := \gamma(\sigma)(e^{-\theta(\sigma)} x )$, 
it follows from \eqref{eq:16} 
that $I(\zeta(\sigma)) = \tilde{I}(\theta(\sigma), \gamma(\sigma))$. 
From this and $\tilde{\gamma} \in \tilde{\Gamma}$, 
one can check that $\zeta \in \Gamma_n$ and 
$c_n \leq \tilde{c}_n$. Thus 
$c_n = \tilde{c}_n$ holds.

	The assertion $d_n \to \infty$ can be proved in a similar way to 
\cite[Lemma 3.2]{HIT-10} (see also \cite{R}) and we skip the details. 
	\end{proof}


\section{Proof of Theorem \ref{definition:1.1}}
\label{section:3}


		In this section, we shall prove Theorem \ref{definition:1.1}. 
We first show the assertion (i). 
To proceed, we first notice that 
	\begin{equation}\label{eq:19}
		\begin{aligned}
			D_u \tilde{I} (\theta, u) \varphi 
			= & e^{N\theta} \int_{\RN} (1 + 4 \pi^2 e^{-2 \theta} |\xi|^2)^\alpha 
			\hat{u} \ov{\hat{\varphi}} \rd \xi - e^{N\theta} \int_{\RN} f(u) \varphi \rd x
			\\
			=& \int_{\RN} ( 1 + 4 \pi^2 |\xi|^2)^\alpha 
			\widehat{u_\theta} \ov{\widehat{\varphi_{\theta}}} \rd \xi 
			- \int_{\RN} f(u_\theta) \varphi_{\theta} \rd x 
			= D_u \tilde{I} (0,u_\theta) \varphi_{\theta}
		\end{aligned}
	\end{equation}
where $\varphi_{\theta} (x) := \varphi(e^{-\theta} x)$.

	\begin{proposition}\label{definition:3.1}
Suppose that $((\theta_n,u_n))_{n=1}^\infty \subset \R \times \Har$ is a Palais--Smale 
sequence of $\tilde{I}$, namely, 
$\tilde{I}(\theta_n,u_n) \to c \in \R$ and 
$D_{(\theta,u)} \tilde{I}(\theta_n,u_n) \to 0$ strongly in $\R \times (\Har)^\ast$. 
Moreover, assume that $(\theta_n)$ is bounded. 
Then $(u_n)$ is bounded in $\Har$. 
	\end{proposition}

	\begin{proof}
Since $(\theta_n)$ is bounded, from \eqref{eq:17} and \eqref{eq:19}, we may assume 
$\theta_n = 0$ by replacing $u_n(x)$ by $u_n(e^{-\theta_n} x)$. 
Therefore, we have 
	\[
		\begin{aligned}
			c + o(1) = \tilde{I}(0,u_n) &=
			 \frac{1}{2} \| u_n \|_\alpha^2 
			- \int_{\RN} F(u_n) \rd x ,
			\\
			o(1) = D_\theta \tilde{I}(0,u_n) 
			&= \frac{N}{2} \| u_n \|_\alpha^2 
			- N \int_{\RN} F(u_n) \rd x 
			- 4\pi^2 \alpha \int_{\RN} (1+4\pi^2|\xi|^2)^{\alpha-1} |\xi|^2 |\wh{u_n}|^2  \rd \xi. 
		\end{aligned}
	\]
From these, it follows that 
	\begin{equation}\label{eq:20}
		\text{$ \displaystyle 
		\left( \int_{\RN} (1+4\pi^2|\xi|^2)^{\alpha-1} |\xi|^2 |\wh{u_n}|^2 \rd \xi 
		\right)_{n=1}^\infty$ is bounded.}
	\end{equation}
Thus if 
	\begin{equation}\label{eq:21}
		\text{$\displaystyle (\| u_n \|_{L^2}^2)_{n=1}^\infty 
		= \left( \int_{\RN} |\wh{u_n}|^2 \rd \xi \right)_{n=1}^\infty$ is bounded}, 
	\end{equation}
then from $(1+4\pi^2 |\xi|^2)^{\alpha-1} \leq 1$ and \eqref{eq:20} we get 
	\[
		\begin{aligned}
			\| u_n \|_\alpha^2 
			&= \int_{\RN} (1+4\pi^2|\xi|^2)^{\alpha-1} ( 1 + 4\pi^2|\xi|^2) 
			|\wh{u_n}(\xi)|^2 \rd \xi 
			\\
			&\leq \int_{\RN} |\wh{u_n}|^2 \rd \xi + 
			4 \pi^2 \int_{\RN} 
			(1+4\pi^2|\xi|^2)^{\alpha-1} |\xi|^2 |\wh{u_n}|^2 \rd \xi
		\end{aligned}
	\]
and $(u_n)$ is bounded in $\Ha$.

	Now we prove \eqref{eq:21} by contradiction and suppose that 
$\tau_n := \| u_n \|_{L^2}^{-2/N} \to 0$ as $n \to \infty$. 
Set $v_n(x) := u_n( \tau_n^{-1} x)$ and observe that 
	\begin{equation}\label{eq:22}
		\| v_n \|_{L^2}^2 = 1, \quad 
		\int_{\RN} |\xi|^{2\alpha} |\wh{v_n}|^2 \rd \xi 
		= \tau_n^{N-2\alpha} \int_{\RN} |\xi|^{2\alpha} |\wh{u_n}|^2 \rd \xi. 
	\end{equation}
Since there exist $C_1,C_2>0$ such that 
	\[
		C_1 |\xi|^2 \leq (1+4\pi^2|\xi|^2)^{\alpha-1}|\xi|^2 
		\quad {\rm if} \ |\xi| \leq 1, \quad 
		C_2 (1 + 4\pi^2 |\xi|^2)^\alpha 
		\leq (1 + 4\pi^2 |\xi|^2)^{\alpha-1} |\xi|^2 
		\quad {\rm if} \ |\xi| \geq 1,
	\]
we observe from \eqref{eq:20} that the quantities 
	\begin{equation}\label{eq:23}
		\text{$ \displaystyle \left( \int_{|\xi| \geq 1}  (1 + 4\pi^2 |\xi|^2)^\alpha 
		|\wh{u_n}|^2 \rd \xi \right)_{n=1}^\infty  
		\ {\rm and} \ 
		\left( \int_{|\xi| \leq 1} |\xi|^2|\wh{u_n}|^2 \rd \xi \right)_{n=1}^\infty
		$ are bounded.}
	\end{equation}
Thus we infer that  
	\[
		\int_{|\xi| \leq 1} | \wh{u_n} |^2 \rd \xi \to \infty. 
	\]
Now we divide our arguments into three steps. First we show 

\medskip

\noindent
{\bf Step 1:} {\sl When $N \geq 3$, 
$v_n \rightharpoonup 0$ weakly in $\Har$. }

\medskip

We first see that for $q  > 1$, by H\"older's inequality and \eqref{eq:23}, 
we obtain 
	\[
		\begin{aligned}
			\int_{|\xi| \leq 1} |\xi|^{2\alpha} |\wh{u_n}|^2 \rd \xi 
			&\leq \left( \int_{|\xi| \leq 1} |\xi|^2 |\wh{u_n}|^2 \rd \xi \right)^{1/q} 
			\left( \int_{|\xi| \leq 1} |\xi|^{2 (\alpha q - 1)/ (q-1)  } |\wh{u_n}|^2 \rd \xi
			 \right)^{1-1/q}
			\\
			& \leq C_q \left( \int_{|\xi| \leq 1} |\xi|^{2 (\alpha q - 1)/ (q-1)  } 
			|\wh{u_n}|^2 \rd \xi \right)^{1-1/q}.
		\end{aligned}
	\]
Choosing $q = \alpha^{-1} \in (1,\infty)$, one gets 
	\[
		\int_{|\xi| \leq 1} |\xi|^{2\alpha} |\wh{u_n}|^2 \rd \xi  
		\leq C_\alpha
		\left( \int_{|\xi| \leq 1} | \wh{u_n} |^2 \rd \xi  \right)^{1-\alpha} 
		\leq C_\alpha \tau_n^{-N(1-\alpha)}. 
	\]
Thus it follows from \eqref{eq:22}, \eqref{eq:23} and $N \geq 3 > 2 \alpha$ that 
	\[
		\begin{aligned}
			\int_{\RN} |\xi|^{2\alpha } |\wh{v_n}|^2 \rd \xi 
			&= \tau_n^{N-2\alpha} 
			\left( \int_{|\xi| \leq 1} + \int_{|\xi| \geq 1} \right) 
			|\xi|^{2\alpha} | \wh{u_n}|^2 \rd \xi 
			\\
			&\leq C_\alpha \tau_n^{(N-2)\alpha} 
			+ \tau_{n}^{N-2\alpha} \int_{|\xi| \geq 1} |\xi|^{2\alpha} 
			| \wh{u_n} |^2 \rd \xi 
			\to 0.
		\end{aligned}
	\]
By Sobolev's inequality $\| u \|_{L^{2^\ast_\alpha}} \leq C \| |\xi|^{\alpha} \wh{u} \|_{L^2}$, 
we obtain $v_n \to 0$ strongly in $L^{2_\alpha^\ast}(\RN)$. 
Thus $v_n \rightharpoonup 0$ weakly in $\Har$.

\medskip

\noindent
{\bf Step 2:} 
{\sl $v_n \rightharpoonup 0$ weakly in $\Har$ when $N = 2$.}

\medskip

	Let $\zeta_0 \in C^\infty_0(\RN)$ satisfy 
$0 \leq \zeta_0 \leq 1$, $\zeta_0(\xi) = 1$ for $|\xi| \leq 1$ and 
$\zeta_0(\xi) = 0$ for $|\xi| \geq 2$. 
Set 
	\[
		\zeta_n(\xi) := \zeta_0(\tau_n \xi), \quad 
		w_{n,1}(x) := \scrF^{-1} \left( \zeta_n(\xi) \wh{v_n}(\xi) \right), 
		\quad 
		w_{n,2} (x) := \scrF^{-1} \left( (1 - \zeta_n(\xi) ) \wh{v_n}(\xi) \right). 
	\]
Then one sees from the Plancherel theorem and 
\eqref{eq:23} that $v_n = w_{n,1} + w_{n,2}$ and 
	\begin{equation}\label{eq:24}
		\begin{aligned}
		& \| w_{n,1} \|_{L^2}^2 = \| \wh{w_{n,1}} \|_{L^2}^2 
		= \int_{\R^2} \zeta_n^2 |\wh{v_n}|^2 \rd \xi 
		\leq \| v_n \|_{L^2}^2 \leq 1, 
		\\
		&
		\| w_{n,2} \|_{L^2}^2 = 
		\tau_n^2 \int_{|\xi| \geq 1} (1-\zeta_0(\xi))^2 |\wh{u_n}|^2 \rd \xi 
		\to 0,
		\\
		& \int_{\RN} w_{n,1} \ov{w_{n,2} } \rd x 
		= \int_{\RN} \wh{w_{n,1}} \ov{\wh{w_{n,2}}} \rd \xi 
		= \tau_n^2 \int_{1 \leq |\xi| \leq 2} \zeta_0(\xi) 
		(1 - \zeta_0(\xi)) |\wh{u_n}|^2 \rd \xi \to 0.
		\end{aligned}
	\end{equation}
We also see that  
	\[
		\begin{aligned}
			\int_{\R^2} |\xi|^2 |\wh{w_{n,1}}|^2 \rd \xi 
			&= \tau_n^4 \int_{\R^2} |\xi|^2 \zeta_0(\tau_n \xi)^2 
			|\wh{u_n}(\tau_n \xi) |^2 \rd \xi 
			= \int_{\R^2} |\xi|^2 \zeta_0(\xi)^2 |\wh{u_n}|^2 \rd \xi 
			\leq \int_{|\xi| \leq 2} |\xi|^2 |\wh{u_n}|^2 \rd \xi.
		\end{aligned}
	\]
Thus, by \eqref{eq:23}, $(w_{n,1})$ is bounded in $H^1_{\rm r}(\R^2)$ and suppose that 
$w_{n,1} \rightharpoonup w_0$ weakly in $H^1(\RN)$.

	On the other hand, by 
	\[
		\int_{\R^2} |\xi|^{2\alpha} | \wh{w_{n,2}}|^2 \rd \xi 
		= \tau_n^{2-2\alpha} \int_{\R^2} 
		|\xi|^{2\alpha} (1 -\zeta_0(\xi))^2 |\wh{u_n}(\xi)|^2 \rd \xi 
		\leq \tau_n^{2-2\alpha} \int_{|\xi| \geq 1} 
		|\xi|^{2\alpha} |\wh{u_n}|^2 \rd \xi  \to 0,
	\]
it follows from \eqref{eq:24} that 
$w_{n,2} \to 0$ strongly in $\Ha$. Recalling $v_n = w_{n,1} + w_{n,2}$, 
we get 
	\begin{equation}\label{eq:25}
		v_n \rightharpoonup w_{0} \left(\in H^1(\RN)\right) \quad 
		{\rm weakly\ in}\ \Ha. 
	\end{equation}

	Now let $\varphi \in C_0^\infty(\RN)$ be radial and 
set $\varphi_n(x) := \varphi(\tau_n x)$. 
Noting 
	\[
		\| \tau_n^2 \varphi_n \|_\alpha^2 
		= \tau_n^2 \int_{\R^2} (1 + 4\pi^2|\xi|^2 \tau^2_n)^\alpha 
		| \wh{\varphi} |^2 \rd \xi \to 0
	\]
and using $D_u \tilde{I}(u_n,0) \to 0$, we infer that 
	\begin{equation}\label{eq:26}
		\begin{aligned}
			\int_{\R^2} f(v_n) \varphi \rd x 
			&=\tau_n^2 \int_{\R^2} f(u_n) \varphi_n \rd x 
			= \tau_n^2 \la u_n , \varphi_n \ra_\alpha 
			+ o ( \| \tau_n^2 \varphi_n \|_\alpha) 
			\\
			&= \int_{\R^2} 
			(1 + 4\pi^2 |\xi|^2 \tau_n^2)^\alpha 
			\wh{v_n} \ov{\wh{\varphi}} \rd \xi 
			+ o ( \| \tau_n^2 \varphi_n \|_\alpha  ). 
		\end{aligned}
	\end{equation}
Since $v_n \to w_0$ strongly in $L^p(\RN)$ for $2<p<2_\alpha^\ast$ 
due to \eqref{eq:25} and Lemma \ref{definition:2.1} (ii), 
by Strauss' lemma(\cite[Lemma 2]{S-77} or \cite[Theorem A.I]{BL-83-1}), (f3) 
and $\varphi \in C_0^\infty(\RN)$, one has 
	\begin{equation}\label{eq:27}
		\int_{\R^2} f(v_n) \varphi \rd x \to \int_{\R^2} f(w_0) \varphi \rd x. 
	\end{equation}
On the other hand, since $\wh{\varphi}$ is rapidly decreasing, it follows that 
	\begin{equation}\label{eq:28}
		\lim_{n \to \infty }\int_{\R^2}	(1 + 4\pi^2 |\xi|^2 \tau_n^2)^\alpha 
		\wh{v_n} \ov{\wh{\varphi}} \rd \xi 
		=  \int_{\R^2} \wh{w_0} \ov{\wh{\varphi}} \rd \xi 
		= \int_{\R^2} w_0 \varphi \rd x. 
	\end{equation}
Now by \eqref{eq:26}--\eqref{eq:28}, we finally obtain 
	\[
		\int_{\R^2} f(w_0) \varphi \rd x 
		= \int_{\R^2} w_0 \varphi \rd x
	\]
for all radial $\varphi \in C^\infty_0(\R^2)$, which yields 
	\begin{equation}\label{eq:29}
		f(w_0) - w_0 \equiv 0 \quad {\rm in}\ \R^2. 
	\end{equation}

	On the other hand, from (f2), one may find some $s_1>0$ such that 
$ s( f(s) - s )<0$ for all $|s| \leq s_1$ with $s \neq 0$. 
Thus by \eqref{eq:29}, $w_0 \in H^1_{\rm r}(\RN) \subset C(\RN \setminus \{0\})$ 
and $w_0(x) \to 0$ as $|x| \to \infty$, we conclude that 
$w_0 \equiv 0$ and Step 2 holds due to \eqref{eq:25}. 

\medskip

\noindent
{\bf Step 3:} 
{\sl Conclusion}

\medskip

	Now we derive a contradiction and conclude 
that \eqref{eq:21} holds. 
Since $(v_n)$ is bounded in $\Ha$ from Steps 1 and 2, 
we first remark that 
	\[
		\begin{aligned}
			\| \tau_n^N u_n \|_\alpha^2 
			&=  \int_{\RN} (1 + 4\pi^2 |\xi|^2)^\alpha 
			| \wh{v_n}(\tau_n^{-1} \xi ) |^2 \rd \xi 
			= \tau_n^N \int_{\RN} (1 + 4\pi^2|\xi|^2 \tau_n^2)^\alpha 
			| \wh{v_n} |^2 \rd \xi \to 0. 
		\end{aligned}
	\]
Let $\delta_0>0$ and $s_1>0$ be constants appearing in \eqref{eq:10}. 
It follows from $ 1 = \| v_n \|_{L^2}^2 \leq \tau_n^N \| u_n \|_\alpha^2$ and 
$D_u \tilde{I}(0,u_n) \to 0$ that 
	\begin{equation}\label{eq:30}
		\begin{aligned}
			\delta_0 &= \delta_0 \| v_n \|_{L^2}^2 
			\leq \tau_n^N \| u_n \|_\alpha^2 - (1-\delta_0) \| v_n \|_{L^2}^2
			\\
			&= \int_{\RN} f(u_n) \tau_n^N u_n \rd x + o(1)
			  - (1-\delta_0) \|v_n\|_{L^2}^2
			 \\
			 &= \int_{\RN} f(v_n) v_n - (1-\delta_0) v_n^2 \rd x 
			 + o(1) 
			 \\
			 &\leq \int_{\RN} \left( f(v_n) v_n - (1-\delta_0) v_n^2 \right)_+ \rd x +o(1). 
		\end{aligned}
	\end{equation}
By \eqref{eq:10}, we observe $\left( f(s) s - (1-\delta_0) s^2 \right)_+ = 0$ 
for $|s| \leq s_1$. 
Hence, arguing as in the proof of Lemma \ref{definition:2.2} (v), 
it follows from \eqref{eq:30} and $v_n \rightharpoonup 0$ weakly in $\Har$ 
due to Steps 1 and 2 that 
	\[
		\delta_0 \leq \int_{\RN} \left( f(v_n) v_n - (1-\delta_0) v_n^2 \right)_+ \rd x +o(1) 
		\to 0,
	\]
which is a contradiction. Thus \eqref{eq:21} holds and 
we complete the proof. 
	\end{proof}

Now we prove the existence of critical points of $I$ which 
satisfy the Pohozaev identity $P(u) = 0$ and 
correspond to $c_n$ in \eqref{eq:18}.

	\begin{proposition}\label{definition:3.2}
There exist a sequence $(u_n) \subset \Har$ such that 
$I'(u_n) = 0$, $I(u_n) = c_n$ and $P(u_n) = 0$. 
Especially, \eqref{eq:4} has infinitely many solutions 
satisfying the Pohozaev identity. 

	\end{proposition}

	\begin{proof}
By Lemma \ref{definition:2.5}, we have $c_n = \tilde{c}_n$. 
Hence, there exists a sequence $(\gamma_{n,k}) \subset \Gamma_{n}$ such that 
	\[
		\max_{\sigma \in D_n} \tilde{I}(0,\gamma_{n,k}(\sigma)) 
		= \max_{\sigma \in D_n} I(\gamma_{n,k}(\sigma)) \to \tilde{c}_n. 
	\]
Applying Ekeland's variational principle to $(\gamma_{n,k})$ and $\tilde{I}$, 
there exist $(\theta_{n,k},u_{n,k}) \in \R \times \Har$ such that 
	\[
		{\rm dist}\, \left( (\theta_{n,k},u_{n,k}) , \{0\} \times \gamma_{n,k} (D_n) 
		\right) 
		\to 0, \quad \tilde{I}(\theta_{n,k},u_{n,k}) \to \tilde{c}_n, 
		\quad 
		D_{(\theta,u)} \tilde{I} (\theta_{n,k}, u_{n,k}) \to 0.
	\]
In particular, $\theta_{n,k} \to 0$. Thus by Proposition \ref{definition:3.1}, 
$(u_{n,k})_k$ is bounded in $\Har$.

	Now assume $u_{n,k} \rightharpoonup u_{n,0}$ weakly in $\Har$ and 
$u_{n,k} \to u_{n,0}$ strongly in $L^p(\RN)$ for $2 < p < 2_\alpha^\ast$. 
Let $\delta_0,s_1>0$ be constants in \eqref{eq:10}. 
By (f4), $\sup_{s \in [0,\infty)} (f(s) - (1-\delta_0 )s ) > 0$. 
Since $f(s)$ is odd, we may find an $s_+ >0$ satisfying 
	\[
		f(\pm s_+) - (1-\delta_0) (\pm s_+) =0, 
		\quad f(s)  - (1-\delta_0) s \neq 0 \quad {\rm for}\ s \in (-s_+,s_+) \setminus \{0\}. 
	\]
Set $f_1(s) := f(s)-(1-\delta_0) s$ if $ s \in [-s_+,s_+]$ and 
$f_1(s) := 0$ otherwise, and $f_2(s) : = f(s) - (1-\delta_0) s - f_1(s)$. 
Remark that 
$s f_1(s) \leq 0$ for all $s \in \R$ and 
$f_2(s) s = 0$ for all $s \in [-s_+,s_+]$.

By $D_u \tilde{I}(u_{n,k}) \to 0$, one can check $I'(u_{n,0}) = 0$. 
Moreover, note that a norm defined by 
	\[
		\| u \|^2 := \| u \|_\alpha^2 - (1-\delta_0) \| u \|_{L^2}^2
	\]
is equivalent to $\| \cdot \|_\alpha$. 
Thus, arguing as in Step 3 of Proposition \ref{definition:3.1}
(see also the proof of Lemma \ref{definition:2.2} (v)), 
from the boundedness of $(u_{n,k})$,
$D_u \tilde{I}( \theta_{n_k}, u_{n_k} ) \to 0$, $\theta_{n_k} \to 0$, 
Fatou's lemma to $f_1(u_{n,k})u_{n,k}$, 
properties of $f_i(s)$ ($i=1,2$), $I'(u_{n,0}) = 0$ and 
the weak convergence of $(u_{n,k})$, we observe that   
	\[
		\begin{aligned}
			\limsup_{k\to \infty}\| u_{n,k} \|^2 &=
			\limsup_{k\to \infty} \int_{\RN} f(u_{n,k}) u_{n,k} - (1-\delta_0) u_{n,k}^2 \rd x 
			\\
			& \leq  \limsup_{k\to \infty}\int_{\RN} f_1(u_{n,k}) u_{n,k} \rd x 
			+ \limsup_{k\to \infty} \int_{\RN} f_2(u_{n,k}) u_{n,k} \rd x
			\\
			&\leq \int_{\RN} f_1(u_{n,0}) u_{n,0} \rd x  
			+ \int_{\RN} f_2(u_{n,0}) u_{n,0} \rd x 
			\\
			&= \int_{\RN} f(u_{n,0}) u_{n,0} - (1-\delta_0) u_{n,0}^2 \rd x 
			= \| u_{n,0} \|^2 \leq \liminf_{k\to \infty} \| u_{n,k} \|^2.
		\end{aligned}
	\]
This implies that $u_{n,k} \to u_{n,0}$ strongly in $\Har$. 
Therefore, $I(u_{n,k}) \to \tilde{c}_n = c_n = I(u_{n,0})$ and $I'(u_{n,0}) = 0$. 
Moreover, recalling \eqref{eq:17}, we have 
	\[
		\lim_{k \to \infty} 
		D_\theta \tilde{I}(\theta_{n,k},u_{n,k}) \to 0 
		= D_\theta \tilde{I}(0,u_{n,0}) = P(u_{n,0}). 
	\]
This completes the proof. 
	\end{proof}

By Proposition \ref{definition:3.2}, a set 
	\[
		S := \left\{ u \in \Ha \ |\ 
		u \not \equiv 0, \ I'(u) = 0, \ P(u) = 0   \right\}
	\]
is not empty. Moreover, we have $c_{\rm LES} \leq c_1$. 
Next we show

	\begin{proposition}\label{definition:3.3}
For every $u \in \Ha$ with $P(u) = 0$ and $u \not \equiv 0$, 
a path $\gamma_u(t) := u(x/t)$ for $t>0$ and $\gamma_u(0) := 0$ 
satisfies 
	\[
		\gamma_u \in C([0,\infty), \Ha), \quad 
		I(u) > I(\gamma_u(t)) \quad {\rm for\ any} \ t \neq 1, \quad 
		I(\gamma_u(t)) \to - \infty \quad {\rm as}\ t \to \infty.
	\]
	\end{proposition}

	\begin{proof}
For $t>0$, one sees 
	\[
		\| \gamma_u(t) \|_{L^2}^2 = t^N \| u \|_{L^2}^2, \quad 
		\| |\xi|^\alpha \wh{\gamma_u(t)} \|_{L^2}^2 
		= t^{N-2\alpha} \int_{\RN} |\xi|^{2\alpha} | \wh{u} |^2 \rd \xi. 
	\]
Thus $\gamma_u \in C([0,\infty), \Ha)$. Furthermore, it follows 
from $I(\gamma_u(t)) = \tilde{I} ( \log t , u  )$ and \eqref{eq:17} that 
	\[
		\begin{aligned}
			\frac{\rd}{\rd t} I(\gamma_u(t)) 
			&= D_\theta\tilde{I} (\log t, u) \frac{1}{t} 
			\\
			&= t^{N-1} 
			\bigg\{ \frac{N}{2}
			\int_{\RN} 
			\left( 1 + 4\pi^2  \frac{|\xi|^2}{t^2} \right)^\alpha |\wh{u}|^2 \rd \xi 
			- N \int_{\RN} F(u) \rd x 
			\\
			& \qquad  \qquad 
			- \alpha \int_{\RN} \left( 1 + 4 \pi^2 \frac{|\xi|^2}{t^2} \right)^{\alpha-1} 
			4\pi^2
			\frac{|\xi|^2}{t^2} |\wh{u}|^2 \rd \xi \bigg\}
			\\
			&=t^{N-1} 
			\bigg\{  \frac{N-2\alpha}{2} \int_{\RN} 
			\left( 1 + 4\pi^2  \frac{|\xi|^2}{t^2} \right)^\alpha |\wh{u}|^2 \rd \xi 
			- N \int_{\RN} F(u) \rd x 
			\\
			& \qquad \qquad + \alpha \int_{\RN}
			\left( 1 + 4\pi^2  \frac{|\xi|^2}{t^2} \right)^{\alpha-1} |\wh{u}|^2 \rd \xi 
			 \bigg\}. 
		\end{aligned}
	\]
Now set 
	\[
		g(t) := \frac{N-2\alpha}{2} \int_{\RN} 
					\left( 1 + 4\pi^2  \frac{|\xi|^2}{t^2} \right)^\alpha |\wh{u}|^2 \rd \xi 
					- N \int_{\RN} F(u) \rd x
				+ \alpha \int_{\RN}
				\left( 1 + 4\pi^2  \frac{|\xi|^2}{t^2} \right)^{\alpha-1} |\wh{u}|^2 \rd \xi .
	\]
Then we get 
	\[
		\begin{aligned}
			g'(t) &= 
			\frac{N-2\alpha}{2} \alpha \int_{\RN} 
			\left( 1 + 4\pi^2  \frac{|\xi|^2}{t^2} \right)^{\alpha-1} 
			(-2) \frac{4\pi^2 |\xi|^2}{t^3} |\wh{u}|^2 \rd \xi 
			\\
			& \qquad 
			+ \alpha (\alpha -1) \int_{\RN} 
			\left( 1 + 4\pi^2  \frac{|\xi|^2}{t^2} \right)^{\alpha-2} 
			(-2) \frac{4\pi^2 |\xi|^2}{t^3} |\wh{u}|^2 \rd \xi
			\\
			&= -2 \alpha 
			\int_{\RN} \left( 1 + 4\pi^2  \frac{|\xi|^2}{t^2} \right)^{\alpha-2} 
			\frac{4\pi^2|\xi|^2}{t^3} |\wh{u}|^2 
			\left\{ \frac{N-2\alpha}{2} \left( 1 + 4\pi^2  \frac{|\xi|^2}{t^2} \right) 
			+ (\alpha-1)  \right\} 
			\rd \xi 
			\\
			&= -2 \alpha \int_{\RN} 
			\left( 1 + 4\pi^2  \frac{|\xi|^2}{t^2} \right)^{\alpha-2} 
			\frac{4\pi^2 |\xi|^2}{t^3} |\wh{u}|^2 
			\left\{ \left( \frac{N}{2} - 1 \right) + \frac{2\pi^2(N-2\alpha) |\xi|^2}{t^2}  \right\} 
			\rd \xi.
		\end{aligned}
	\]
Since $N \geq 2$ and $N > 2 \alpha$, one has  
$g'(t) < 0$ for all $t > 0$. Noting that 
$\rd I(\gamma_u(t))/ \rd t = t^{N-1} g(t)$ and that 
$P(u) = 0$ is equivalent to $g(1) = 0$, we see that 
	\[
		\frac{\rd}{\rd t} I(\gamma_u(t)) > 0 \quad {\rm if} \ 0 < t < 1, \quad 
		\frac{\rd}{\rd t} I(\gamma_u(t)) < 0 \quad {\rm if} \ 1 < t,
	\]
which implies that $I(\gamma_u(t))$ has a unique maximum at $t=1$. 
From the monotonicity of $g(t)$ and $g(1) = 0$, 
it is clear that $I(\gamma_u(t)) \to - \infty$ as $t \to \infty$ and 
Proposition \ref{definition:3.3} holds. 
	\end{proof}

	Before proceeding to a proof of $c_1 = c_{\rm LES}$, 
we define the following quantities:  
	\begin{equation}\label{eq:31}
	\begin{aligned}
		c_{\rm MP,r} &:= \inf_{\gamma \in \Gamma_{\rm r}} \max_{0 \leq t \leq 1} 
		I (\gamma(t)), \qquad 
		\Gamma_{\rm r} := \left\{ \gamma \in C([0,1], \Har \ |\ 
		\gamma(0) = 0, \ I(\gamma(1)) < 0 \right\},
		\\ 
		c_{\rm MP} &:= \inf_{\gamma \in \Gamma} \max_{0 \leq t \leq 1} 
		I (\gamma(t)), 
		\qquad 
		\Gamma := \left\{ \gamma \in C([0,1], \Ha \ |\ 
		\gamma(0) = 0, \ I(\gamma(1)) < 0 \right\}. 
		\\
		\tilde{c}_{\rm MP,r} 
		&:= \inf_{\tilde{\gamma} \in \tilde{\Gamma}_{\rm r}} 
		\max_{ 0 \leq t \leq 1 } 
		\tilde{I} ( \tilde{\gamma}(t) ), 
		\\
		\tilde{\Gamma}_{\rm r} &:= 
		\left\{ \tilde{\gamma}=(\theta, \gamma) \in C([0,1], \R \times \Har ) \ 
		|\ \gamma \in \Gamma_{\rm r}, \ 
		\theta(0) = 0 = \theta(1)  \right\}.
		\end{aligned}
	\end{equation}
Then we show 
	\begin{lemma}\label{definition:3.4}
		$0 < c_{\rm MP} = c_{\rm MP,r}= c_1 = c_{\rm LES} 
		= \tilde{c}_{\rm MP,r}$. 
	\end{lemma}

	\begin{proof}
As in the proof of Lemma \ref{definition:2.5}, one sees 
$\tilde{c}_{\rm MP,r} = c_{\rm MP,r}$. 
Moreover, from the definition, $c_{\rm MP} \leq c_{\rm MP,r} \leq c_1$. 
In addition, by Propositions \ref{definition:3.2} and \ref{definition:3.3}, 
we know $ c_{\rm MP} \leq c_{\rm LES} \leq c_1$. 
Thus it is sufficient to prove $c_1 \leq c_{\rm MP}$.

	We first claim that 
	\begin{equation}\label{eq:32}
		c_{\rm MP} = d := \inf_{\eta \in \ov{\Gamma} } 
		\max_{ 0 \leq t \leq 1} 
		I(\eta(t)), \quad 
		\ov{\Gamma}
		:= \left\{ \eta \in C([0,1], \Ha) \ |\ 
		\eta(0) = 0, \ \eta(1) = \gamma_1(1)  \right\}
	\end{equation}
where $\gamma_1$ appears in Lemma \ref{definition:2.3}. 
From the definition of $\Gamma$ and $I(\gamma_1(1)) < 0$, 
we have $c_{\rm MP} \leq d$. 
For the opposite inequality $d \leq c_{\rm MP}$, 
it is enough to show that $[I<0]:=\{ u \in \Ha \ |\ I(u) < 0 \}$ 
is path-connected in $\Ha$. 
A similar claim is proved in \cite{HIT-10} for the case $\alpha = 1$ 
and we use the same argument.

		Let $u_1,u_2 \in [I<0]$. Since $C_0^\infty(\RN)$ is dense in $\Ha$, 
we may assume $u_1,u_2 \in C_0^\infty(\RN)$. 
For $u_i$, we consider the path $\gamma_i (t) = \gamma_{u_i} (t)$ 
appearing in Proposition \ref{definition:3.3}. 
From the computations in the proof of Proposition \ref{definition:3.3}, 
we observe that $\rd I(\gamma_i(t)) / \rd t >0$ if $0< t \ll 1$. 
Since $I(\gamma_i(0)) = 0 > I(\gamma_i(1)) =I(u_i)$, 
there are maximum points $t_i \in (0,1)$ of $I(\gamma_i(t))$ with 
$I(\gamma_i(t_i))>0$. 
At those points, we have $\rd I( \gamma_i(t) ) / \rd t |_{t=t_i} = 0$, 
which yields $P(\gamma_{i}(t_i)) = 0$. 
Hence, by Proposition \ref{definition:3.3}, we observe that 
$t \mapsto I(\gamma_i(t)): (t_i,\infty) \to \R$ is strictly decreasing and 
$I(\gamma_i(t)) \to - \infty$ as $t \to \infty$. 
Thus choose $t_0>1$ so large that 
	\[
		I(\gamma_i(t_0)) < -2  \max\left\{  \max_{0 \leq s \leq 1} 
		I( s u_1) , \ 
		\max_{0 \leq s \leq 1  } I( s u_2) \right\} < 0. 
	\]
Noting that $\gamma_i(t_0)(x) = u_i(x/t_0)$, $u_i \in C^\infty_0(\RN)$ and 
	\[
		\begin{aligned}
			\la u_1(x/t_0) , s u_2(x - R\mathbf{e}_1  ) \ra_\alpha 
			&\to 0, \\
			\int_{\RN} F( u_1(x/t_0) + s u_2( x-R\mathbf{e}_1 ) ) \rd x 
			&\to \int_{\RN} F(u_1/t_0) \rd x + \int_{\RN} F(s u_2(x)) \rd x
		\end{aligned}
	\]
uniformly with respect to $s \in [0,1]$ as $R \to \infty$ 
where $\mathbf{e}_1 = (1,0,\ldots, 0)$,
it follows from the choice of $t_0$ that as $R \to \infty$ 
	\[
		\max_{ 0 \leq s \leq 1} 
		I\left(\gamma_1(t_0) + s u_2 (x - R\mathbf{e}_1) \right)  
		\to I(\gamma_1(t_0)) + \max_{0 \leq s \leq 1} I(su_2) 
		< - \max_{0 \leq s \leq 1} I(s\gamma_2(t_0))  < 0.
	\]
Hence, choosing $R_0>0$ so large, we have 
	\[
		\begin{aligned}
			&{\rm supp}\, \gamma_1(t) \cap {\rm supp}\, u_2 ( \cdot - R_0 \mathbf{e}_1) 
			= \emptyset \quad {\rm for} \ 1 \leq t \leq t_0, \quad 
			\max_{0 \leq s \leq 1} 
			I\left(\gamma_1(t_0) + s u_2( \cdot - R_0\mathbf{e}_1) \right)  < 0,
			\\
			& I ( \gamma_1(t) + u_2 (\cdot - R_0 \mathbf{e}_1  ) )  
			= I(\gamma_1(t)) + I(u_2) < 0 \quad {\rm for} \ 1 \leq t \leq t_0.
		\end{aligned}
	\]

		Through the paths $t \mapsto \gamma_1(t)$ ($ t\in[1,t_0]$), 
$s \mapsto \gamma_1(t_0)(x) + s u_2 (x - R_0 \mathbf{e}_1)$ 
($s \in [0,1]$) and $\theta \mapsto \gamma_1(t_0 - \theta) (x) + u_2(x - R_0\mathbf{e}_1)$ 
($\theta \in [0,t_0-1]$), 
we can connect $u_1$ and $u_1(x) + u_2(x-R_0 \mathbf{e}_1) $ 
in $[ I < 0]$. 
In a similar fashion, we see that there is a path 
between $u_2$ and $u_1(x) + u_2(x-R_0 \mathbf{e}_1) $ 
in $[ I < 0]$. Therefore, $[I<0]$ is path-connected in $\Ha$ and 
$ c_{\rm MP} = d$ follows.

	Next, since $I(u)=I(-u)$, $\gamma_1(1) \in \Har$ and 
$\gamma(-t) = - \gamma(t)$ 
holds for any $\gamma \in \Gamma_1$, 
it suffices to prove 
	\begin{equation}\label{eq:33}
		d= \inf_{\eta \in \ov{\Gamma}_{\rm r } }
				\max_{ 0 \leq t \leq 1} 
				I(\eta(t))  (=c_1), \quad 
		\ov{\Gamma}_{\rm r}
				:= \left\{ \eta \in C([0,1], \Har) \ |\ 
				\eta(0) = 0, \ \eta(1) = \gamma_1(1)  \right\}.
	\end{equation}
By definition, one has $d \leq c_1$.

		On the other hand, let $\eta \in \ov{\Gamma}$. 
Since $F(s)$ is even, we have $\int_{\RN} F(u) \rd x = \int_{\RN} F(|u|) \rd x$ 
for each $u \in \Ha$. 
Moreover, noting the following inequality (see Lemma \ref{definition:A.3}) 
	\begin{equation}\label{eq:34}
		\| |u| \|_\alpha \leq \| u \|_\alpha \quad 
		{\rm for\ all} \ u \in \Ha, 
	\end{equation}
we observe that $\zeta(t)(x) := |\eta(t)(x)|$ satisfies 
	\[
		I(\zeta(t)) \leq I(\eta(t)) \quad {\rm for\ all} \ t \in [0,1], \quad 
		\zeta \in C([0,1], \Ha). 
	\]
Recalling Remark \ref{definition:2.4}, we have $\zeta \in \ov{\Gamma}$.

	Now set 
$\tilde{\zeta}(t)(x) := (\zeta(t))^\ast (x)$
where $u^\ast$ denotes the Schwarz symmetrization of $u$. 
By \cite[Theorem 9.2]{AL-89}, we have $\tilde{\zeta} \in C([0,1], \Har)$ 
and $\int_{\RN} F( \tilde{\zeta}(t) ) \rd x = \int_{\RN} F(\zeta (t)) \rd x$. 
Moreover, from Remark \ref{definition:2.4}, 
it follows that $(\gamma_1(1))^\ast = \gamma_1(1)$, 
which yields $\tilde{\zeta} \in \ov{\Gamma}_{\rm r}$. 
Since $\| \tilde{\zeta}(t) \|_\alpha \leq \| \zeta(t) \|_\alpha$ holds 
thanks to \cite[Proposition 4]{Se}, we see that 
	\[
		c_1 \leq \max_{0 \leq t \leq 1} I(\tilde{\zeta}(t)) \leq 
		\max_{0 \leq t \leq 1} I(\zeta(t)) \leq \max_{0 \leq t \leq 1} I(\eta(t)). 
	\] 
Since $\eta$ is any element of $\ov{\Gamma}$, one sees 
$c_1 \leq d$.

Therefore, from \eqref{eq:32} and \eqref{eq:33}, we get $c_{\rm MP} = d =c_1$ 
and this completes the proof. 
	\end{proof}

Finally we shall show that if either $\alpha > 1/2$ or 
$f(s)$ is locally Lipschitz, then 
every weak solution of \eqref{eq:4} satisfies $P(u) = 0$. 
To this end, we use the following Br\'ezis--Kato type result \cite{BK-79}:

	\begin{proposition}\label{definition:3.5}
		Assume that $u \in \Ha$ is a weak solution of 
		\begin{equation}\label{eq:35}
			(1-\Delta)^\alpha u - a(x) u = 0 \quad {\rm in} \ \RN
		\end{equation}
		where $a(x)$ satisfies 
		\begin{equation}\label{eq:36}
			|a(x)| \leq C_0 ( 1 + A(x)) \quad {\rm for\ a.e.} \ x \in \RN, \quad 
			A \in L^{N/(2\alpha)}(\RN).
		\end{equation}
		Then $u \in L^p(\RN)$ for all $p \in [2,\infty)$. 
	\end{proposition}
	
We give a sketch of proof for Proposition \ref{definition:3.5} in Appendix. 
For related results, see \cite{FF-15}. 
Using Proposition \ref{definition:3.5}, we shall prove
	\begin{proposition}\label{definition:3.6}
		\emph{(i)} 
		Let $u \in \Ha$ be a weak solution of \eqref{eq:4}. 
		Then $u \in C^{\beta}_{\rm b}(\RN)$ for every $\beta \in (0,2\alpha)$ and 
		$u$ decays faster than any polynomial , i.e., 
		for any $k \in \N$ there exists a $c_k >0$ such that 
		$|u(x)| \leq c_k (1 + |x|)^{-k}$ for all $x \in \RN$. 
		Here 
			\[
				\begin{aligned}
					C^\beta_{\rm b}(\RN ) &:= \left\{ u \in C(\RN) \cap L^\infty(\RN) \ \Big|\ 
					[u]_{C^\beta} := 
					\sup_{x,y \in \RN, x \neq y} 
					\frac{|u(x)-u(y)|}{|x-y|^{\beta}} < \infty  \right\} 
					\quad {\rm if} \ \beta < 1,
					\\
					C^{1}_{\rm b}(\RN) &:= \{  u \in C^1(\RN) \ |\ u, \nabla u \in L^\infty(\RN) \}
					,
					\\
					C^{\beta}_{\rm b} (\RN) &:= \{ u \in C^1_{\rm b}(\RN) \ |\ 
					\nabla u \in C^{\beta - 1}_{\rm b}(\RN) \} 
					\quad {\rm if} \ 1 < \beta < 2.
				\end{aligned}
			\]

		\emph{(ii)} 
		In addition to \emph{(f1)--(f4)}, assume that $f(s)$ is locally Lipschitz continuous 
		and $0 < \alpha \leq 1/2$. 
		Then $ u \in C^{1+\beta}_{\rm b}(\RN)$ for all $\beta \in (0,2\alpha)$. 
		
		\emph{(iii)} 
		Assume that $u \in \Har \cap C^1_{\rm b}(\RN)$ is a weak solution of \eqref{eq:4}. 
		Then $P(u) = 0$. 
	\end{proposition}

A proof of (iii) below is essentially due to \cite{FV-15}. See also \cite{Se}.

	\begin{proof}
(i) Let $u \in \Ha$ be a weak solution of \eqref{eq:4} and 
set $a(x) := f(u(x))/u(x)$. From (f2) and (f3), it follows that 
	\[
		|a(x)| \leq C( 1 + |u(x)|^{2^\ast_\alpha - 2} ) \quad 
		{\rm for \ a.e.} \ x \in \RN. 
	\]
Noting that $u$ is a weak solution of $(1-\Delta)^\alpha u - a(x) u = 0$ in $\RN$, 
Proposition \ref{definition:3.5} yields $u \in L^p(\RN)$ for all $p \in [2,\infty)$. 
Using (f2) and (f3) again, we observe that $f(u(x)) \in L^p(\RN)$ 
for any $p \in [2,\infty)$. Thus, recalling the argument in Lemma \ref{definition:2.1}, 
we have 
	\[
		u = G_{2\alpha} \ast f(u), \quad 
		u \in \scrL^p_{2\alpha} := \left\{ 
		G_{2\alpha} \ast g \ |\ g \in L^p(\RN)  \right\} 
		\quad {\rm for\ all} \ p \in [2,\infty). 
	\]
Since $\scrL^p_{2\alpha} \subset W^{2\alpha,p}(\RN)$ 
(see \cite{St}), Sobolev's inequality yields 
$u \in C^{\beta}_{\rm b}(\RN)$ for each $\beta \in (0,2\alpha)$. 
Hence, the first assertion in (i) holds.

	For the decay estimate, let $\delta_0>0$ be a constant 
appearing in \eqref{eq:10} and $v \in \Ha$ a unique weak solution of 
	\[
		(1 - \Delta )^\alpha v - (1-\delta_0) v = \left( f(u) - (1-\delta_0) u  \right)_+
		=:g(x)
		\quad {\rm in}\ \RN.
	\]
Since $u \in C^\beta_{\rm b}(\RN)$, we have $u(x) \to 0$ as $|x| \to \infty$ and 
$g(x) \in L^\infty(\RN)$ has compact support. 
Hence, by Proposition \ref{definition:A.1} (ii), for any $k \in \N$, $v$ satisfies 
$v(x) \leq c_k ( 1 + |x|)^{-k}$ for $x \in \RN$. Moreover, 
rewriting \eqref{eq:4} as 
	\[
		(1-\Delta)^\alpha u - (1-\delta_0) u = f(u) - (1-\delta_0) u \quad 
		{\rm in} \ \RN,
	\]
Proposition \ref{definition:A.1} (iii) asserts 
$u(x) \leq v(x) \leq c_k (1+|x|)^{-k}$ for all $x \in \RN$.

	On the other hand, let $w$ be a solution of 
	\[
		(1-\Delta )^\alpha w - (1-\delta_0) w = \left( f(u) - (1-\delta_0) u\right)_-
		\quad {\rm in} \ \RN.
	\]
Since 
	\[
		(1-\Delta)^\alpha (-u) - (1-\delta_0) (-u) 
		= - ( f(u) - (1-\delta_0) u  ) \leq (1-\Delta )^\alpha w - (1-\delta_0) w
		\quad {\rm in} \ \RN,
	\]
using Proposition \ref{definition:A.1} again, we get 
$-u(x) \leq w(x) \leq c_k (1+|x|)^{-k}$ for $x \in \RN$. 
Thus (i) holds.

(ii) Let $f(s)$ be locally Lipschitz and $0 < \alpha \leq 1/2$. 
Since $u \in C^\beta_{\rm b}(\RN)$ for $0 < \beta < 2\alpha$ by (i), 
we have 
	\[
		|f(u(x_1)) - f(u(x_2))| \leq L |u(x_1) - u(x_2) | \leq L [u]_{C^\beta} |x_1-x_2|^\beta.
	\]
Thus $f(u(x)) \in C^{\beta}_{\rm b}(\RN)$. 
Since $u = G_{2 \alpha} \ast f(u)$ and 
$C^{\beta + 2\alpha}_{\rm b}(\RN) = G_{2\alpha} \ast C^{\beta}_{\rm b}(\RN)$ 
holds by \cite[Theorem 4 in $\S$5 of Chapter V]{St}, 
applying the bootstrap argument, we can check that 
$u \in C^{\beta}_{\rm b}(\RN)$ for all $\beta < 1 + 2 \alpha$.

	(iii) We follow the argument in \cite{FV-15,Se}. 
Let $u \in \Ha \cap C^1_{\rm b}(\RN)$ be a weak solution of \eqref{eq:4}. 
For a mollifier $(\rho_\e)$, set $u_\e(x) := u \ast \rho_\e$. 
Thanks to the decay estimate of $u$, we observe that $ u_\e \in \scrS(\RN,\R)$. 
Thus, 
	\begin{equation}\label{eq:37}
		\la u, x \cdot \nabla u_\e \ra_\alpha = \int_{\RN } f(u) 
		x \cdot \nabla u_\e \rd x. 
	\end{equation}
Noting 
	\[
		\begin{aligned}
			&\la u , x \cdot \nabla u_\e \ra_\alpha 
			= \int_{\RN} \wh{u} (\xi) \ov{ (1+4\pi^2|\xi|^2)^\alpha \wh{x\cdot \nabla u_\e} } \rd \xi 
			= \int_{\RN} u(x) (1-\Delta )^\alpha (x\cdot \nabla u_\e) \rd x,
			\\
			& (1-\Delta)^\alpha u_\e 
			= (1-\Delta)^\alpha( \rho_\e \ast u ) 
			= \rho_\e \ast (1-\Delta)^\alpha u 
			= \rho_\e \ast f(u),
		\end{aligned}
	\]
and using (\cite[Proposition 5.1]{FV-15})
	\[
		(1-\Delta)^\alpha (x \cdot \nabla u_\e )
		= x \cdot \nabla [ (1-\Delta)^\alpha u_\e ] 
		+ 2 \alpha (1 -\Delta)^\alpha u_\e 
		- 2\alpha (1 - \Delta)^{\alpha - 1} u_\e,
	\]
we obtain 
	\begin{equation}\label{eq:38}
		\begin{aligned}
			\la u , x\cdot \nabla  u_\e \ra_\alpha 
			&= \int_{\RN} u(x) (1-\Delta)^\alpha (x \cdot \nabla u_\e) 
			\\
			&= \int_{\RN} 
			u(x) \left[  x \cdot \nabla ( \rho_\e \ast f(u)  ) 
			+ 2\alpha \rho_\e \ast f(u) - 2\alpha (1-\Delta)^{\alpha-1} u_\e \right] 
			\rd x.
		\end{aligned}
	\end{equation}
By $u_\e \to u$ strongly in $L^2(\RN)$ and $0 < \alpha < 1$, 
it is easily seen that 
	\begin{equation}\label{eq:39}
		\begin{aligned}
			&\lim_{\e \to 0} \int_{\RN} u \rho_\e \ast f(u) \rd x 
			= \lim_{\e \to 0} \int_{\RN} (u \ast \rho_\e) f(u) \rd x 
			= \int_{\RN} u f(u) \rd x = \| u \|_\alpha^2,
			\\
			&\lim_{\e \to 0} \int_{\RN} (1-\Delta)^{\alpha-1} u_\e u \rd x 
			= \lim_{\e \to 0} \int_{\RN} (1 + 4\pi^2|\xi|^2)^{\alpha-1} 
			\wh{u_\e} \ov{\wh{u}} \rd \xi 
			= \int_{\RN} (1+4\pi^2|\xi|^2)^{\alpha-1} |\wh{u}|^2 \rd \xi.
		\end{aligned}
	\end{equation}

	On the other hand, recalling $u \in C^1_{\rm b}$, we have 
	\[
		\begin{aligned}
			\int_{\RN} u x \cdot \nabla ( \rho_\e \ast f(u)) \rd x 
			&= \sum_{i=1}^N \int_{\RN} u x_i \partial_{x_i} ( \rho_\e \ast f(u) ) \rd x 
			= - \sum_{i=1}^N \int_{\RN} \partial_{x_i} ( u x_i) \rho_\e \ast f(u) \rd x
			\\
			&= - N \int_{\RN } u \rho_\e \ast f(u) \rd x 
			- \sum_{i=1}^N \int_{\RN} 
			x_i (\partial_{x_i} u) \rho_\e \ast f(u) \rd x. 
		\end{aligned}
	\]
From the decay estimate of $u$, the same estimate holds for 
$\rho_\e \ast f(u)$ uniformly with respect to $\e$. Thus, 
letting $\e \to 0$ in the above equality, the dominated convergence 
theorem and \eqref{eq:4} give us 
	\begin{equation}\label{eq:40}
		\begin{aligned}
			\lim_{\e \to 0} \int_{\RN} u x \cdot \nabla ( \rho_\e \ast f(u)) \rd x 
			&= - N \int_{\RN} u f(u) \rd x 
			- \sum_{i=1}^{N} \int_{\RN} x_i (\partial_{x_i} u) f(u) \rd x 
			\\
			&= - N \| u \|_\alpha^2 - \sum_{i=1}^N \int_{\RN} 
			x_i \partial_{x_i} F(u) \rd x 
			\\
			&= - N \| u\|_\alpha^2 + N \int_{\RN} F(u) \rd x. 
		\end{aligned}
	\end{equation}

	Finally, since $ \nabla u_\e(x) \to \nabla u(x)$ in $L^\infty_{\rm loc}(\RN)$ 
and $(\nabla u_\e)$ is bounded in $L^\infty(\RN)$ 
due to $u \in C^1_{\rm b}(\RN)$, 
one sees 
	\begin{equation}\label{eq:41}
		\lim_{\e \to 0} \int_{\RN} f(u) x \cdot \nabla u_\e \rd x 
		= \int_{\RN} f(u) x \cdot \nabla u \rd x 
		= \int_{\RN} x \cdot \nabla F(u) \rd x 
		= - N \int_{\RN} F(u) \rd x. 
	\end{equation}
Therefore, collecting \eqref{eq:37}--\eqref{eq:41}, we obtain 
	\[
		0 = (N-2\alpha) \| u \|_\alpha^2 - 2N \int_{\RN} F(u) \rd x 
		+ 2 \alpha \int_{\RN} (1 + 4\pi^2 |\xi|^2)^{\alpha-1} 
		|\wh{u}|^2 \rd \xi = 2 P(u).
	\]
Thus we complete the proof. 
	\end{proof}

	\begin{proof}[Proof of Theorem \ref{definition:1.1}]
From Propositions \ref{definition:3.2}, 
\ref{definition:3.3},  \ref{definition:3.6} and Lemmas \ref{definition:2.5} and 
\ref{definition:3.4}, the only task is to show that 
there is a positive solution $u_1$ of \eqref{eq:4} which corresponds 
to the value $c_1$.

	First,  select $(\gamma_n) \subset \Gamma_{\rm r}$ so that 
$\max_{0 \leq t \leq 1} I(\gamma_n(t)) \to c_{\rm MP,r}$. 
As in the proof of Lemma \ref{definition:3.4}, setting $\eta_n(t)(x) := |\gamma_n(t)|(x)$, 
by Lemma \ref{definition:A.3}, we observe that 
$I(\eta_n(t)) \leq I(\gamma_n(t))$ for every $t \in [0,1]$ and 
$\eta_n \in \Gamma_{\rm r}$. Hence, it follows from Lemma \ref{definition:3.4} that 
	\[
		\max_{0 \leq t \leq 1} \tilde{I}(0,\eta_n(t)) \to \tilde{c}_{\rm MP,r} = c_1. 
	\]
Thus as in the proof of Proposition \ref{definition:3.2}, 
applying the Eklenad's variational principle to $\tilde{I}$ and $(\eta_n)$, 
and noting $\eta_n(t)(x) \geq 0$, 
we may find $(\theta_n)$ and $(u_n) \subset \Har$ so that 
	\[
		I(u_n) \to c_1, \quad D_{(\theta,u)} \tilde{I}(\theta_n,u_n) \to 0, \quad 
		\theta_n \to 0, \quad 
		(u_n)_-(x) := \max \{ 0, - u(x) \} \to 0 \quad {\rm in} \ \Ha. 
	\]
Repeating the argument of Proposition \ref{definition:3.2}, 
$u_n \to u_0$ strongly in $\Ha$, $I(u_0) = c_{1}$, $u_0 \geq 0, \not\equiv 0$. 
Using Propositions \ref{definition:3.5} and \ref{definition:3.6}, 
we have $u \in C^\beta_{\rm b}(\RN)$ for any $\beta \in (0,2\alpha)$. 
Finally, a weak Harnack inequality (\cite[Proposition 2]{FF-15}) 
gives $u_{1,0} > 0$ in $\RN$ and we complete the proof. 
\end{proof}


\section{Proof of Theorem \ref{definition:1.3}}
\label{section:4}


	In this section, we shall prove Theorem \ref{definition:1.3} using 
Theorem \ref{definition:1.1}.  
Until the proof of Theorem \ref{definition:1.3}, we consider the more general setting. 
Indeed, we first assume that $f(x,s)$ in \eqref{eq:1} satisfies the following:
	\begin{enumerate}
		\item[(G1)] 
		$f \in C(\RN \times \R, \R)$ and $f(x,-s) = - f(x,s)$ for each $x \in \RN$ and 
		$s \in \R$. 
		\item[(G2)] 
			\[
				-\infty < \liminf_{s \to 0} \inf_{x \in \RN} \frac{f(x,s)}{s} 
				\leq \limsup_{s \to 0} \sup_{x \in \RN} \frac{f(x,s)}{s} < 1. 
			\]
		\item[(G3)] 
			\[
				\lim_{|s| \to \infty} 
				\sup_{x \in \RN} \frac{|f(x,s)|}{|s|^{2_\alpha^\ast-1}} = 0.
			\]
		\item[(G4)] 
			There exists an $s_0>0$ such that 
				\[
					\inf_{x \in \RN} 
					\left( F(x,s_0) - \frac{1}{2}s_0^2 \right) > 0.
				\]
		\item[(G5)] 
			There exists a $f_\infty(s) \in C(\R,\R)$ such that 
			$f(x,s) \to f_\infty(s)$ in $L^\infty_{\rm loc}(\R)$ as $|x| \to \infty$. 
	\end{enumerate}
Note that (G1)--(G3) and (G5) are weaker than 
(F1)--(F5). Moreover, we remark that (F5) implies (G4). 
Indeed, by the inequalities in (F5), we have 
	\[
		G(x,s) \geq 
		\left( \frac{s}{s_1} \right)^\mu G(x,s_1) \geq c_1 s^{\mu}
		\quad {\rm for\ all} \ (x,s) \in \RN\times [s_1,\infty)
	\]
for some $c_1>0$. 
Since $\mu>2$ and $V \in L^\infty(\RN)$, we can easily find an 
$s_0>0$ such that 
	\[
		F(x,s_0) - \frac{1}{2}s_0^2 
		= G(x,s_0) - \frac{V(x)+1}{2} s_0^2 \geq c_1 s_0^\mu 
		- \frac{V(x) + 1}{2} s_0^2> 0. 
	\]
Thus (G4) is derived from (F5).

	Under (G1)--(G5), we define the functional $J$ corresponding to \eqref{eq:1}:
	\[
		J(u) := \frac{1}{2} \| u \|_\alpha^2 
		- \int_{\RN} F(x,u(x)) \rd x \in C^1(\Ha, \R).
	\]
Remark that 
a critical point of $J$ is equivalent to a solution of \eqref{eq:1}. 
We begin with showing that $J$ has the mountain pass geometry:
	\begin{proposition}\label{definition:4.1}
		Under \emph{(G1)--(G5)}, there exist $\rho_0>0$ and $u_1 \in \Ha$ 
		such that 
			\[
				\inf_{\| u \|_\alpha = \rho_0} J(u) > 0, \quad 
				J(u) \geq 0 \quad {\rm if}\ \| u \|_\alpha \leq \rho_0, \quad 
				J(u_1) < 0. 
			\]
	\end{proposition}

	\begin{proof}
By (G2) and (G3), we find $\delta_0>0$ and $s_1>0$ such that  
	\[
		\sup_{x \in \RN} \frac{f(x,s)}{s} 
		\leq 1 - 2 \delta_0 \quad {\rm for\ all} \ |s| \leq s_1. 
	\]
Hence, 
	\[
		\sup_{x \in \RN}  F(x,s) \leq \frac{1-\delta_0}{2}s^2 
		\quad {\rm for\ every}\ |s| \leq s_1. 
	\]
Combining this with (G3), we have 
	\[
		\left( F(x,s) - \frac{1-\delta_0}{2} s^2 \right)_+ \leq C |s|^{2^\ast_\alpha} 
		\quad {\rm for \ each} \ x \in \RN, \ s \in \R. 
	\]
Hence, Sobolev's inequality yields 
	\[
		\begin{aligned}
			J(u) &= \frac{1}{2} \| u \|_\alpha^2 - \frac{1-\delta_0}{2} \| u \|_{L^2}^2 
			- \int_{\RN} \left( F(x,s) - \frac{1-\delta_0}{2} u^2 \right) \rd x
			\\
			& \geq \frac{\delta_0}{2} \| u \|_\alpha^2 
			- \int_{\RN}  \left( F(x,s) - \frac{1-\delta_0}{2} u^2 \right)_+ \rd x 
			\geq \frac{\delta_0}{2} \| u \|_\alpha^2 
			- C \| u \|_\alpha^{2^\alpha_\ast}.
		\end{aligned}
	\]
Choosing $\rho_0>0$ sufficiently small, we have 
	\[
		\inf_{\| u \|_\alpha = \rho_0} J(u) > 0, \quad 
		J(u) \geq 0 \quad {\rm if}\ \| u \|_\alpha \leq \rho_0. 
	\]

	For the existence of $u_1$, 
let us consider a function defined by 
	\[
		u_1(x) := \left\{ \begin{aligned}
			& s_0 & &{\rm if}\ |x| \leq R,\\
			& -s_0(|x|-R) + s_0 & &{\rm if}\ R < |x| \leq R+1,\\
			& 0 & &{\rm if} \ |x| > R +1
		\end{aligned}  \right.
	\]
where $s_0>0$ appears in (G4). Notice that $ u_1 \in H^1(\RN)$. 
By $\| u \|_\alpha \leq \| u \|_{H^1}$ for all $u \in H^1(\RN)$, we obtain 
	\[
		J(u_1) \leq \frac{1}{2}\|u_1 \|_{H^1}^2 - \int_{\RN} F(x,u_1) \rd x
		= \frac{1}{2} \| \nabla u_1 \|_{L^2}^2 
		- \int_{\RN} F(x,u_1) - \frac{1}{2} u^2_1 \rd x.
	\]
Since it is easy to check $\| \nabla u_1 \|_{L^2}^2 =O(R^{N-2})$ 
and $\int_{\RN} F(x,u_1) - \frac{1}{2} u_1^2 \rd x \geq c R^N + O(R^{N-1})$ as 
$R \to \infty$ for some $c>0$, for sufficiently large $R>0$, one finds that 
	\[
		J(u_1) \leq \frac{1}{2} \| \nabla u_1 \|_{L^2}^2 
				- \int_{\RN} F(x,u_1) - \frac{1}{2} u^2_1 \rd x < 0. 
	\]
Thus we complete the proof. 
	\end{proof}

	By Proposition \ref{definition:4.1}, we define the mountain pass value of $J$: 
	\[
		d_{\rm MP} := \inf_{\gamma \in \Gamma_J} \max_{0 \leq t \leq 1} 
		J(\gamma(t)) > 0, \quad 
		\Gamma_{J} := \left\{ \gamma \in C([0,1], \Ha) \ |\ 
		\gamma(0)=0, \ J(\gamma(1)) < 0 \right\}.
	\]
As in the proof of Theorem \ref{definition:1.1}, 
applying Ekeland's variational principle to $J$ and 
a sequence of paths $(\gamma_n(t)) \subset \Gamma_J$ 
where $\gamma_n(t) (x) \geq 0$ and $\max_{0 \leq t \leq 1} J(\gamma_n(t)) \to d_{\rm MP}$, 
we may find a Palais--Smale sequence $(v_n)$ of $J$ at level $d_{\rm MP}$: 
	\begin{equation}\label{eq:42}
			J(v_n) \to d_{\rm MP}, \quad 
			J(v_n) \to 0 \quad {\rm strongly \ in\ } \ (\Ha)^\ast, \quad 
			(v_n)_- \to 0 \quad 
			{\rm strongly \ in} \ \Ha. 
	\end{equation}

	We first observe the behaviors of bounded Palais--Smale sequneces 
of $J$ under (G1)--(G5). For this purpose, consider 
	\begin{equation}\label{eq:43}
		(1 - \Delta )^\alpha \omega = f_\infty(\omega) \quad {\rm in}\ \RN, \quad 
		\omega \in \Ha
	\end{equation}
and define the functional $J_\infty$ corresponding to \eqref{eq:43} by 
	\[
		J_\infty (\omega) := \frac{1}{2} \| \omega \|_\alpha^2 
		- \int_{\RN} F_\infty ( \omega) \rd x
	\]
where $F_\infty(s) := \int_0^s f_\infty(t) \rd t$. 
Notice that $f_\infty$ satisfies (f1)--(f4) thanks to (G1)--(G5).  
Hence, $J_\infty \in C^1(\Ha,\R)$ and critical points of $J$ 
are solutions of \eqref{eq:43}. 

	\begin{proposition}\label{definition:4.2}
		Suppose that \emph{(G1)--(G5)} hold and 
		that every weak solution of \eqref{eq:43} satisfies 
		the Pohozaev identity $P_\infty(u) = 0$ where 
			\[
				P_\infty (u) := \frac{N-2\alpha}{2} \| u \|_\alpha^2 
				- N \int_{\RN} F_\infty (u) \rd x 
				+ \alpha \int_{\RN} (1 + 4 \pi^2 |\xi|^2)^{\alpha-1} |\wh{u}|^2 \rd \xi. 
			\]
		Let $(u_n)$ be a bounded Palais--Smale sequence of $J$, that is, 
		$(u_n) \subset \Ha$ is bounded and satisfies 
			\[
				J(u_n) \to c \in \R, \quad 
				J'(u_n) \to 0 \quad {\rm strongly \ in}\ (\Ha)^\ast.
			\]
		Then there exist $\ell \in \N$, $u_0 \in \Ha$, 
		$\omega_i \in \Ha$ and $(y_{n,i})_{n=1}^\infty$ 
		for $i=1,\ldots, \ell$ provided $\ell \geq 1$ 
		such that 
		
		\emph{(i)} $|y_{n,i}| \to \infty$ for $1 \leq i \leq \ell$ and 
		$|y_{n,i} - y_{n,j}| \to \infty$ for $ i \neq j$ as $n \to \infty$. 
		
		\emph{(ii)} $J'(u_0) = 0$, $\omega_i \not\equiv 0$ and 
		$J_\infty'(\omega_i) = 0$ for $1 \leq i \leq \ell$. 
		
		\emph{(iii)} When $\ell \geq 1$, 
			\[
				\left\| u_n - u_0 - \sum_{i=1}^\ell \omega_i(\cdot - y_{n,i}) 
				\right\|_\alpha \to 0, 
				\quad 
				c = \lim_{n\to \infty} J(u_n) = I(u_0) + \sum_{i=1}^\ell J_\infty (\omega_i)
			\]
		When $\ell = 0$, $\| u_n - u_0 \|_\alpha \to 0$ and 
		$c = J(u_0)$. 
	\end{proposition}

We postpone a proof of Proposition \ref{definition:4.2} and 
prove it after the proof of Theorem \ref{definition:1.3}.

Next, since $f_\infty$ satisfies (f1)--(f4) under (G1)--(G5), 
we may define the mountain pass value of $J_\infty$: 
	\[
		d_\infty := \inf_{\gamma \in \Gamma_{J_\infty}} 
		\max_{0 \leq t \leq 1} J_\infty(\gamma(t)) > 0, \quad 
		\Gamma_{J_\infty} := 
		\left\{ \gamma \in C([0,1], \Ha) \ |\ \gamma(0) =0, \ 
		J_\infty(\gamma(1)) < 0 \right\}. 
	\]

	\begin{proposition}\label{definition:4.3}
Let \emph{(G1)--(G5)} hold. Furthermore, suppose that 

\smallskip

\noindent
\emph{(I)} $F_\infty(s) \leq F(x,s) $ for all $(x,s) \in \RN \times \R$. 

\noindent
\emph{(II)} There exists a bounded Palais--Smale sequence $(u_n)$ of $J$ at 
level $d_{\rm MP}$ with $(u_n)_- \to 0$ strongly in $\Ha$. 

\noindent
{\rm (III)} Every weak solution of \eqref{eq:43} satisfies the Pohozaev identity 
$P_\infty(u) = 0$. 

\smallskip

\noindent
Then, \eqref{eq:1} admits a positive solution. 
	\end{proposition}

	\begin{proof}
By (I), we observe that 
$J(u) \leq J_\infty(u)$ for all $u \in \Ha$. 
Hence, $\Gamma_{J_\infty} \subset \Gamma_J$ and 
$0 < d_{\rm MP} \leq d_{\infty}$ hold. 
We divide our case into two cases: 

\medskip

\noindent
{\bf Case 1:} {\sl $d_{\rm MP} < d_\infty$ holds.}

\medskip  

In this case, we apply Proposition \ref{definition:4.2} for $(u_n)$ to obtain 
$u_0, \omega_1, \ldots, \omega_\ell \in \Ha$ and 
$(y_{n,i})_{n=1}^\infty$ ($i=1,\ldots, \ell$) satisfying (i)--(iii) 
in Proposition \ref{definition:4.2}. 
Remark that $J'(u_0) = 0$ and $u_0$ is a weak limit of $(u_n)$ 
(see the proof of Proposition \ref{definition:4.2} below). 
Arguing as in the proof of Proposition \ref{definition:3.6}, 
we observe that 
$u_0 \in C^{\beta}_{\rm b} (\RN)$ and $u_0$ is nonnegative. 
Therefore, if $u_0 \not\equiv 0$, then 
a weak Harnack inequality (\cite[Proposition 2]{FF-15}) implies 
that $u_0$ is the desired solution of \eqref{eq:1}. 
Thus, it suffices to show that the case 
$u_0 \equiv 0$ does not happen.

	To this end, let $u_0 \equiv 0$ and we may assume $\ell \geq 1$ 
thanks to $d_{\rm MP} > 0$. Since $\omega_i \not \equiv 0$ 
and $J_\infty'(\omega_i) = 0$, Theorem \ref{definition:1.1} and 
the assumption (III) assert that $ J_\infty(\omega_i) \geq d_\infty > 0$ 
for all $1 \leq i \leq \ell$. Thus Proposition \ref{definition:4.2} (iii) and $u_0 \equiv 0$ 
yield
	\[
		d_\infty > d_{\rm MP} = \sum_{i=1}^\ell J_\infty(\omega_i) 
		\geq \ell d_\infty.
	\]
Since $d_\infty > 0$, this is a contradiction. 
Hence, $u_0 \equiv 0$ does not occur in this case. 

\medskip

\noindent
{\bf Case 2:} {\sl $d_{\rm MP} = d_\infty$ holds.  }

\medskip

	In this case, by Theorem \ref{definition:1.1}, we find 
a positive solution $\omega \in \Ha$ of \eqref{eq:43} and 
a path $\gamma_\omega \in \Gamma_{J_\infty}$ so that 
	\begin{equation}\label{eq:44}
		\begin{aligned}
			& \omega(0) = \| \omega \|_{L^\infty}, 
			\quad \omega \in \gamma_\omega ([0,1]), \quad 
			\| \omega \|_{L^\infty} = 
			\| \gamma_\omega (t) \|_{L^\infty} 
			\  {\rm for}\ t \in (0,1], \\
			&
			J_\infty (\omega) = \max_{0 \leq t \leq 1} J_\infty(\gamma_\omega(t)) 
			= d_{\infty}, \quad 
			J_\infty(\gamma_\omega(t)) < J_\infty(\omega) \quad 
			{\rm if} \ \gamma_\omega(t) \neq \omega.
		\end{aligned}
	\end{equation}
Noting $\gamma_\omega(t)(\cdot - z) \in \Gamma_{J_\infty} \subset \Gamma_J$ 
for any $z \in \RN$, let 
$t_z \in (0,1)$ satisfy $\max_{0 \leq t \leq 1} J(\gamma_\omega(t)(\cdot - z)) 
= J( \gamma_\omega (t_z) (\cdot - z) )$. 
Then we get 
	\begin{equation}\label{eq:45}
		d_\infty = d_{\rm MP} \leq J(\gamma_\omega (t_z) (\cdot - z) ) 
		\leq J_\infty (\gamma_\omega (t_z) (\cdot - z) ) 
		\leq  J_\infty(\omega) = d_{\infty}. 
	\end{equation}
From \eqref{eq:44} and \eqref{eq:45}, 
we deduce that $\gamma_\omega(t_z) = \omega$ and 
	\[
		\int_{\RN} F(x, \omega (x - z )) \rd x 
		= \int_{\RN} F_\infty(\omega (x-z)) \rd x 
		\quad {\rm for\ any}\ z \in \RN.
	\]
Recalling $F_\infty (s) \leq F(x,s)$ for any $(x,s) \in \RN \times \R$, we observe that 
	\[
		F(x, s) = F_\infty(s) \quad {\rm for \ all} \ (x,s) \in \RN \times 
		\left[0,\| \omega \|_{L^\infty}\right].
	\]
This implies $f(x,s) = f_\infty (s)$ for all $x \in \RN \times [0,\| \omega \|_{L^\infty}]$. 
From this fact, we see that $\omega$ is also a positive solution of \eqref{eq:1}. 
Thus we complete the proof. 
	\end{proof}

	\begin{remark}\label{definition:4.4}
(i) From the above proof and the existence of optimal path in Theorem \ref{definition:1.1}, 
we have $d_{\rm MP} < d_\infty$ when $F(x,s) \leq F_\infty(s)$ and 
$F(x,s) \not \equiv  F_\infty(s)$ for each $s \in \R$. 

(ii)
In the case $\alpha = 1$, the Pohozaev identity is useful 
to obtain a bounded Palais--Smale sequence. For instance, we refer to \cite{AP-09,JT-05}. 
When $0<\alpha < 1$, in addition to (G1)--(G4), assume that $f(x,s)$ is differentiable in $x$, 
$\nabla_x f(x,s) \in C(\RN \times \R, \RN )$ and for each $M>0$, 
there exists a $C_M > 0$ such that 
$|\nabla_x f(x,s) | \leq C_M $ for all $(x,s) \in \RN \times [-M,M]$. 
Under these conditions, if $u \in \Ha \cap C^1_{\rm b}(\RN)$ is a weak solution of \eqref{eq:1}, 
then $u$ satisfies the following Pohozaev identity: 
	\begin{equation}\label{eq:46}
		0 = \frac{N-2\alpha}{2} \| u \|_\alpha^2 
		- N \int_{\RN} F(x,u) \rd x 
		- \int_{\RN} (x \cdot \nabla_x F) (x,u) \rd x 
		+ \alpha \int_{\RN} (1 + 4 \pi^2 |\xi|^2 )^{\alpha - 1} 
		| \wh{u} |^2 \rd \xi .
	\end{equation}
In fact, \eqref{eq:46} can be proved by following the argument of 
Proposition \ref{definition:3.6} (iii) and noting 
	\[
		\sum_{i=1}^N x_i \partial_{x_i} u f(x,u) 
		= \sum_{i=1}^N x_i 
		\left\{ \partial_{x_i} \left( F(x,u) \right)  
		- (\partial_{x_i} F(x, u))  \right\} 
		= (x \cdot \nabla_x)  F(x,u) - 
		(x \cdot \nabla_x F) (x, u). 
	\]
As in the case $\alpha = 1$, the Pohozaev identity \eqref{eq:46} may be 
useful to get a bounded Palais--Smale sequence.

	\end{remark}

Now we prove Theorem \ref{definition:1.3}.

	\begin{proof}[Proof of Theorem \ref{definition:1.3}]
Let us assume (F1)--(F5). As we have already seen, 
(G1)--(G5) also holds. Moreover, conditions (I) and (III) in Proposition \ref{definition:4.3} 
follow from (F4) and Theorem \ref{definition:1.1}. 
Thus we only need to check that (II) holds. 
For this purpose, we use (F5) to show that $(v_n)$ in \eqref{eq:42} is bounded 
in $\Ha$ and the argument is standard (for instance, see \cite{R}). From (F5), we have 
	\begin{equation}\label{eq:47}
		\begin{aligned}
			\mu d_{\rm MP} + o(1) + o(1) \| v_n \|_\alpha 
			&\geq \mu J(v_n) - J'(v_n) v_n 
			\\
			&= \frac{\mu-2}{2} 
			\left( \| v_n \|_\alpha^2 + \int_{\RN} V(x) v_n^2 \rd x \right) 
			- \int_{\RN} \mu G(x,v_n) - g(x,v_n)v_n \rd x
			\\
			& \geq 
			\frac{\mu-2}{2} 
			\left( \| v_n \|_\alpha^2 + \int_{\RN} V(x) v_n^2 \rd x \right).
		\end{aligned}
	\end{equation}
Since $\inf_{\RN} V > -1$ due to (F2), a quantity defined by 
	\[
		\| u \|^2 := \| u \|_\alpha^2 + \int_{\RN} V(x) u^2 \rd x 
	\]
is an equivalent norm to $\| \cdot \|_\alpha$. 
Therefore, from \eqref{eq:47}, we infer that 
$(v_n)$ is bounded in $\Ha$ and 
Proposition \ref{definition:4.3} implies Theorem \ref{definition:1.3}. 
	\end{proof}

		Now we turn to prove Proposition \ref{definition:4.2}. 
We first recall the following lemma due to 
\cite[Lemma 2.1]{FQT-12} (cf. \cite[Lemma 2.18]{CR-92}, \cite[Lemma I.1]{L-84-2}
and \cite[Lemma 3.1]{TWY-12}):

\begin{lemma}\label{definition:4.5}
	Let $(u_n) \subset \Ha$ be a bounded and satisfy 
		\[
			\sup_{y \in \Z^N} \int_{y + Q} |u_n|^p \rd x \to 0 
			\quad {\rm for \ some} \ p \in [2,2_\alpha^\ast), \quad 
			Q := [0,1]^N.
		\]
	Then $u_n \to 0$ strongly in $L^q(\RN)$ for all $q \in (2,2^\ast_\alpha)$. 
\end{lemma}
This lemma is proved in \cite[Lemma 2.1]{FQT-12}, 
however, for the sake of readers, we show it here.

	\begin{proof}
First we note that by the boundedness of $(u_n)$ and 
the interpolation inequality, we may assume $2<p$ without loss of generality. 
Next, we shall prove the existence of $C_0>0$ satisfying 
		\begin{equation}\label{eq:48}
			\| u \|_{L^p(z+Q)} \leq C_0 \| u \|_{W^{2,\alpha}(z+Q)} \quad 
			{\rm for\ all}\ z \in \Z^N, \ u \in W^{2,\alpha}(z+Q)
		\end{equation}
where 
	\[
		W^{2,\alpha}(\Omega) := \{  u \in L^2(\Omega)\ |\ [u]_{W^{2,\alpha}(\Omega)} < \infty
		\}, \quad 
		[u]_{W^{2,\alpha}(\Omega)}^2 := 
		\int_{\Omega \times \Omega} 
		\frac{|u(x)-u(y)|^2}{|x-y|^{N+2\alpha}} \rd x \rd y 
	\]

	We first consider it on $Q$. For $u \in W^{2,\alpha}(Q)$, let 
	\[
		\tilde u (x',x_N) := \left\{\begin{aligned}
			& u(x',x_N) & &{\rm if}\ x_N \geq 0,\\
			& u(x',-x_N) & &{\rm if}\ x_N < 0.
		\end{aligned}\right.
	\]
Set $Q_1 := Q \cup (Q - \mathbf{e}_N)$ and $R x := (x',-x_N)$ 
where $\mathbf{e}_N :=(0,\ldots,0,1)$. 
Since $|x-y| = |Rx - Ry|$ and $|x-y| \leq |Rx - y|$ 
for all $x,y \in Q$, and $\tilde u (Rx) = u(x)$ for $x \in Q$, 
it is easy to see that 
	\[
		\begin{aligned}
			\| \tilde u \|_{L^2(Q_1)}^2 &= 2 \| u \|_{L^2(Q)}^2,\\
			 [ \tilde u]_{W^{2,\alpha}(Q_1)}^2 
			&= \int_{Q_1 \times Q_1} \frac{|\tilde u (x) - \tilde u(y)|^2  }
			{|x-y|^{N+2\alpha}} \rd x \rd y 
			\\
			&= \left( \int_{Q \times Q} + \int_{Q \times (Q - \mathbf{e}_N)} 
			+ \int_{(Q - \mathbf{e}_N) \times Q } 
			+ \int_{(Q - \mathbf{e}_N) \times (Q - \mathbf{e}_N)}   \right)
			\frac{|\tilde u (x) - \tilde u(y)|^2  }
						{|x-y|^{N+2\alpha}} \rd x \rd y
						\\
			& \leq 4 \int_{Q\times Q} \frac{| u (x) -  u(y)|^2  }
						{|x-y|^{N+2\alpha}} 
			= 4 [u]_{W^{2,\alpha}(Q)}^2.
		\end{aligned}
	\]
Thus $\tilde u \in W^{2,\alpha}(Q_1)$ and 
$\| \tilde u \|_{W^{2,\alpha}(Q_1)} \leq 2 \| u \|_{W^{2,\alpha}(Q)}$.

	Repeating the above argument $2N-1$ times for edges of $Q_1$ except for 
$x_N = -1$, there exists an $Eu \in W^{2,\alpha}(Q_{2N})$ such that 
	\[
		Eu = u \quad {\rm on} \ Q, \quad 
		\| Eu \|_{W^{2,\alpha}(Q_{2N})} \leq 2^{2N} \| u \|_{W^{2,\alpha}(Q)}
	\]
where $Q_{2N}$ is a cube satisfying 
$[-1,2]^{N} \subset Q_{2N}$. Choosing a smooth domain $\Omega \subset \RN$ such that 
$ Q \subset \Omega \subset [-1,2]^N$, it follows from Sobolev's embedding 
(\cite[Theorems 5.6 and 6.7]{DNGE-12}) that 
	\[
		\| u \|_{L^p(Q)} \leq \| Eu \|_{L^p(\Omega)} 
		\leq C_\Omega \| Eu \|_{W^{2,\alpha}(\Omega)} 
		\leq C_\Omega \| Eu \|_{W^{2,\alpha}(Q_{2N})} \leq 
		2^{2N} C_\Omega \| u \|_{W^{2,\alpha}(Q)}
	\]
for all $u \in W^{2,\alpha}(Q)$. 
For \eqref{eq:48}, it is enough to translate $Q$, $Q_{2N}$ and $\Omega$. 
Thus \eqref{eq:48} holds.

	Now we complete the proof. We first notice from \cite[Proposition 3.4]{DNGE-12} 
that 
	\[
		[ u ]_{W^{2,\alpha}(\RN)}^2 = C(N,\alpha) \int_{\RN} |\xi|^{2\alpha} 
		| \hat{u} (\xi) |^2 \rd \xi. 
	\]
Thus, it is easily seen that 
	\[
		\begin{aligned}
			\sum_{z \in \Z^N} [u]_{W^{2,\alpha}(z+Q)}^2 
			&= \sum_{z \in \Z^N} 
			\int_{(z+Q) \times (z+Q)} \frac{| u (x) -  u(y)|^2  }{|x-y|^{N+2\alpha}} 
			\rd x \rd y
			\\
			\\
			& \leq \sum_{z \in \Z^N} 
			\int_{\RN \times (z+Q)} \frac{| u (x) -  u(y)|^2  }{|x-y|^{N+2\alpha}} 
			\rd x \rd y = [u]_{W^{2,\alpha}(\RN)}^2 \leq C \| u \|_{\alpha}^2. 
		\end{aligned}
	\]
Therefore, by \eqref{eq:48} and $p>2$, we obtain 
	\[
		\begin{aligned}
			\| u_n \|_{L^p(\RN)}^p &= \sum_{z \in \Z^N} \| u_n \|_{L^p(z+Q)}^p 
			\leq \left( \sup_{z \in \Z^N} \| u_{n} \|_{L^p(z+Q)}^{p-2} \right) 
			\sum_{z \in \Z^N} \| u_n \|_{L^p(z+Q)}^{2}
			\\
			&\leq \left( \sup_{z \in \Z^N} \| u_{n} \|_{L^p(z+Q)}^{p-2} \right) 
			\sum_{z \in \Z^N} C_0 \| u \|_{W^{2,\alpha}(z+Q)}^2 
			\\
			& \leq C_0 \left( \sup_{z \in \Z^N} \| u_{n} \|_{L^p(z+Q)}^{p-2} \right) 
			\| u_n \|_{\alpha}^2 \to 0. \qedhere
		\end{aligned}
	\]
	\end{proof}

Now we prove Proposition \ref{definition:4.2}. 

	\begin{proof}[Proof of Proposition \ref{definition:4.2}]
We argue as in \cite[Proof of Proposition 4.2]{JT-04}. 
Since $(u_n)$ is bounded in $\Ha$, 
choosing a subsequence if necessary 
(still denoted by $(u_n)$), we may assume 
$u_n \rightharpoonup u_0$ weakly in $\Ha$. 
From $J'(u_n) \to 0$ strongly in $(\Ha)^\ast$, it is easy to check $J'(u_0)= 0$.

	Next, we claim that there exist $\ell \geq 0$, $\omega_i \not \equiv 0$ 
and $(y_{n,i})_{n=1}^\infty$ for $i = 1, \ldots, \ell$ if $\ell \neq 0$ such that 
properties (i) and (ii) in Proposition \ref{definition:4.2} hold and 
	\begin{equation}\label{eq:49}
		\left\| u_n - u_0 - \sum_{i=1}^\ell \omega_i(\cdot - y_{n,i}) \right\|_{L^p} 
		\to 0 \quad {\rm for \ every}\  p \in (2 , 2_\alpha^\ast).
	\end{equation}
For this purpose, we consider 
	\[
		\limsup_{n\to \infty}\sup_{z \in \Z^N} \int_{z + Q} | u_n - u_0 |^2 \rd x 
		=: c_1.
	\]
If $c_1= 0$, then Lemma \ref{definition:4.5} yields \eqref{eq:49} with $\ell = 0$.

	Next, consider the case $c_1>0$. 
Then we choose $(y_{n,1}) \subset \RN$ such that 
	\[
		\lim_{n \to \infty} \int_{y_{n,1} + Q} 
		|u_n-u_0|^2 \rd x \to c_1 > 0.
	\]
Let $u_n(\cdot + y_{n,1}) \rightharpoonup \omega_1$ 
weakly in $\Ha$. Since $u_n \to u_0$ strongly in $L^2_{\rm loc}(\RN)$ 
and $c_1>0$, we have $|y_{n,1}| \to \infty$ and 
$\omega_1 \not\equiv 0$. Moreover, 
from $|y_{n,1}| \to \infty$ and 
$J'(u_n) [ \varphi (\cdot - y_{n,1})  ] \to 0$ for each $\varphi \in C^\infty_0(\RN)$, 
we also see that $J_\infty'(\omega_1) = 0$ by (G5). 
Since every weak solution of \eqref{eq:43} satisfies the Pohozaev identity 
$P_\infty(u) = 0$ and $f_\infty$ satisfies (f1)--(f4), 
by Theorem \ref{definition:1.1}, we have $J_\infty(\omega_1) \geq d_{\infty} > 0$. 
Choosing a $\zeta_0>0$ so that 
$J_\infty(u) < d_{\infty}$ for all $\| u \|_\alpha < \zeta_0$, 
we obtain $\| \omega_1 \|_\alpha \geq \zeta_0$.

	Next, consider 
		\[
			\limsup_{n \to \infty} \sup_{z \in \Z^N} 
			\int_{z+Q} | u_n - u_0 - \omega_1(x-y_{n,1})|^2 \rd x 
			=: c_2.
		\]
When $c_2 = 0$, then Lemma \ref{definition:4.5} yields \eqref{eq:49}. 
On the other hand, when $c_2>0$, we select a $(y_{n,2}) \subset \RN$ so that 
	\[
		\lim_{n \to \infty} \int_{y_{n,2}+Q} |u_n - u_0 - \omega_1(x-y_{n,1})|^2 \rd x 
		= c_2. 
	\]
Let $u_{n,2}(x+y_{n,2}) - u_0(x+y_{n,2}) - \omega_1(x-y_{n,1}+y_{n,2}) 
\rightharpoonup \omega_2$ weakly in $\Ha$. 
Then as in the above, it is immediate to see that 
	\[
		|y_{n,2}| \to \infty, \ 
		|y_{n,1} - y_{n,2}| \to \infty, \ 
		u_{n}(x + y_{n,2}) \rightharpoonup \omega_2 \not \equiv 0, \ 
		J_\infty'(\omega_2) = 0, 
		\ J_\infty(\omega_2) \geq d_{\infty},\ 
		\| \omega_2 \|_\alpha \geq \zeta_0.
	\]

	Now we repeat the same procedure. Namely, consider 
	\[
		\limsup_{n\to \infty} \sup_{z \in \Z^N} \int_{z + Q} 
		| u_n - u_0 - \omega_1(x-y_{n,1}) - \omega_2(x-y_{n,2})|^2 \rd x 
		= : c_3
	\]
and find $\omega_3 \not \equiv 0$ and $(y_{n,3})$ if $c_3>0$. 
Therefore, we obtain 
$\ell \in \N$, $\omega_i \not\equiv 0$ ($1 \leq i \leq \ell$) and 
$(y_{n,i})_{n=1}^\infty$ ($1 \leq i \leq \ell$) so that 
	\begin{equation}\label{eq:50}
		\begin{aligned}
			& |y_{n,i}| \to \infty, \ |y_{n,i} - y_{n,j}| \to \infty 
			\ {\rm if} \ i \neq j, \ 
			u_n(x+y_{n,i}) \rightharpoonup \omega_i \not\equiv 0, \\
			&J_\infty'(\omega_i) = 0, \ 
			J_\infty (\omega_i) \geq d_{\infty}, \ 
			\| \omega_i \|_\alpha \geq \zeta_0.
		\end{aligned}
	\end{equation}
To prove \eqref{eq:49}, it suffices to prove that this procedure cannot be iterated 
infinitely many times and 
	\begin{equation}\label{eq:51}
		\limsup_{n\to\infty} 
		\sup_{z \in \Z^N} \int_{z + Q} 
		\left| u_n - u_0 - \sum_{i=1}^\ell \omega_i(x-y_{n,i}) \right|^2 \rd x 
		=0 \quad {\rm for\ some} \ \ell \in \N.
	\end{equation}
To see this, from \eqref{eq:50} it follows that 
	\[
		\begin{aligned}
			0 &\leq  \left\| u_n - u_0 - \sum_{i=1}^\ell \omega_i(\cdot - y_{n,i}) \right\|_\alpha^2 
			= \| u_n \|_\alpha^2 + \| u_0 \|_\alpha^2 
			+ \sum_{i=1}^\ell \| \omega_i \|_\alpha^2 
			\\
			& \quad 
			- 2 \la u_n, u_0 \ra_\alpha - 2 \sum_{i=1}^\ell 
			\la u_n , \omega_i(\cdot - y_{n,i}) \ra_\alpha 
			\\
			& \quad 
			+ 2 \sum_{i=1}^\ell \la u_0 , \omega_i(\cdot - y_{n,i}) \ra_\alpha 
			+ 2 \sum_{1 \leq i < j \leq \ell} 
			\la \omega_i ( \cdot - y_{n,i}) , \omega_j ( \cdot - y_{n,j}) \ra_\alpha
			\\
			&= \| u_n\|_\alpha^2 - \| u_0 \|_\alpha^2 
			- \sum_{i=1}^\ell \| \omega_i \|_\alpha^2 + o(1)
			\\
			& \leq \| u_n \|_\alpha^2 - \|u_0\|_\alpha^2 - \ell \zeta_0 + o(1).
		\end{aligned}
	\]
Since $(u_n)$ is bounded, we observe that the above procedure cannot 
be iterated infinitely many times. 
Therefore, \eqref{eq:51} holds, which implies \eqref{eq:49}.

	Finally, we shall prove that  
	\begin{equation}\label{eq:52}
		\left\| u_n - u_0 - \sum_{i=1}^\ell \omega_i(\cdot - y_{n,i}) \right\|_\alpha 
		\to 0.
	\end{equation}
To do this, set $U_n(x) := u_n(x) - u_0(x) - \sum_{i=1}^\ell \omega_i(x-y_{n,i})$. 
By (G2), select $\delta_0>0$ and $s_1>0$ so that 
	\begin{equation}\label{eq:53}
		f(x,s) s \leq (1-\delta_0) s^2 \quad {\rm for\ all} \ 
		(x,s) \in \RN \times [-s_1,s_1].
	\end{equation}
It is clear that a norm
	\[
		\| u \|^2 := \| u \|_\alpha^2 - (1-\delta_0) \|u\|^2_{L^2}
	\]
is equivalent to $\|\cdot\|_\alpha$. Therefore, instead of \eqref{eq:52}, 
we shall show $\| U_n \| \to 0$.

	To this end, putting $f_{\delta_0}(x,s) := f(x,s) - (1-\delta_0)s$ and 
$f_{\infty,\delta_0}(s) := f_\infty(s) - (1-\delta_0)s$, 
we first notice from $J'(u_n) \to 0$, $J'(u_0)=0$ and $J_\infty'(\omega_i)=0$ that 
	\begin{equation}\label{eq:54}
		\begin{aligned}
			\| U_n \|^2 = &\| U_n \|_\alpha^2 - (1-\delta_0) \|U_n\|_{L^2}^2
			\\
			= \,& \la u_n - u_0 - \sum_{i=1}^\ell \omega_i(\cdot - y_{n,i}) , U_n
			\ra_\alpha - (1-\delta_0) \la u_n - u_0 - \sum_{i=1}^{\ell} \omega_i(\cdot - y_{n,i}) ,
			U_n \ra_{L^2}
			\\
			=\,& \int_{\RN} f_{\delta_0}(x,u_n) U_n \rd x - \int_{\RN} f_{\delta_0}(x,u_0) U_n \rd x 
			- \sum_{i=1}^\ell \int_{\RN} f_{\infty,\delta_0}(\omega_i(x-y_{n,i})) U_n \rd x + o(1)
			\\
			=\,&\int_{\RN} \left\{ f_{\delta_0}(x,u_n) - f_{\delta_0}(x,u_0) - \sum_{i=1}^\ell 
			f_{\delta_0}(x, \omega_i(x-y_{n,i}))   \right\} U_n \rd x 
			\\
			& \quad + \sum_{i=1}^\ell
			\int_{\RN} \left\{ 	f(x, \omega_i(x-y_{n,i})) - f_\infty( \omega_i(x-y_{n,i}) ) \right\} 
			U_n \rd x  + o(1)
			\\
			=:\,& I_n + II_n + o(1).
		\end{aligned}
	\end{equation}

		We shall show $I_n=o(1)=II_n$. We first consider $I_n$. 
For any $M > 0$, by H\"older's inequality, we have 
	\begin{equation}\label{eq:55}
		\begin{aligned}
			& \int_{[|U_n| \geq M]} \left| f_{\delta_0}(x,u_n) - f_{\delta_0}(x,u_0) 
			- \sum_{i=1}^{\ell} f_{\delta_0}(x, \omega_i(x-y_{i,n}) ) \right| 
			| U_n | \rd x
			\\
			 \leq \, & \|U_n\|_{L^{2_\alpha^\ast} ([|U_n| \geq M]) } 
			 \Big( \| f_{\delta_0}(x,u_n) \|_{L^{p^\ast} ([|U_n| \geq M]) }  
			 + \| f_{\delta_0}(x,u_0) \|_{L^{p^\ast} ([|U_n| \geq M]) } 
			 \\
			 & \hspace{5cm}
			 + \sum_{i=1}^\ell \| f_{\delta_0} (x, \omega_i(x - y_{n,i}) ) 
			 \|_{L^{p^\ast} ([|U_n| \geq M]) } \Big)
		\end{aligned}
	\end{equation}
where $p^\ast := 2_\alpha^\ast / (2_\alpha^\ast - 1)$. 
Since $(U_n)$ is bounded in $L^{2_\alpha^\ast}(\RN)$ and $p^\ast < 2$, 
we have 
	\[
		C_1 \geq \| U_n \|_{L^{2_\alpha^\ast} ([|U_n| \geq M]) }^{2^\ast_\alpha} 
		\geq M^{2_\alpha^\ast} \mathcal{L}^N( [|U_n| \geq M] ), \quad 
		\| u \|_{L^{p^\ast} ( [ |U_n| \geq M ] ) }^{p^\ast} 
		\leq \mathcal{L}^N( [|U_n| \geq M] )^{1 - p^\ast / 2} 
		\| u \|_{L^2}^{p_\ast}
	\]
where $C_1>0$ is independent of $n$. In particular, 
$ \sup_{n \geq 1} \mathcal{L}^N( [|U_n| \geq M] ) \to 0$ as $M \to \infty$. 
Recalling $|f_{\delta_0}(x,s)| \leq c_\e |s| + \e |s|^{2_\alpha^\ast}$ for all 
$(x,s) \in \RN \times \R$, it follows from H\"older's inequality and 
the boundedness of $(u_n)$ that 
	\[
		\begin{aligned}
			& \sup_{n \geq 1} 
			\left\{ \| f_{\delta_0}(x,u_n) \|_{L^{p^\ast} ([|U_n| \geq M]) }^{p^\ast} 
			+ \| f_{\delta_0}(x,u_0) \|_{L^{p^\ast} ([|U_n| \geq M]) }^{p^\ast} 
			+ \| f_{\delta_0}(x,\omega_i(x - y_{i,n})) \|_{L^{p^\ast} ([|U_n| \geq M]) }^{p^\ast} 
			\right\}
			\\
		\leq \, & \sup_{n \geq 1} \int_{[|U_n| \geq M]} 
		\Big\{ c_\e |u_n|^{p^\ast} + \e |u_n|^{2_\alpha^\ast} 
		+ c_\e |u_0|^{p^\ast} + \e |u_0|^{2_\alpha^\ast}
		\\
		& \qquad \qquad \qquad 
		+ \sum_{i=1}^\ell \left( c_\e |\omega_i(x-y_{n,i})|^{p^\ast} 
		+  \e |\omega_i(x-y_{n,i})|^{2_\alpha^\ast}\right) \Big\} \rd x 
		\\
		\leq \, & c_\e \mathcal{L}^N( [|U_n| \geq M] )^{1 - p^\ast / 2}  + C_2 \e
		\end{aligned}
	\]
for some $C_2>0$. 
Therefore, by \eqref{eq:55} and 
$\sup_{n \geq 1}\mathcal{L}^N( [|U_n| \geq M] ) \to 0$ as $M \to \infty$, we get 
	\[
		\begin{aligned}
			&\limsup_{M \to \infty} 
			\sup_{n \geq 1}\int_{[|U_n| \geq M]} \left| f_{\delta_0}(x,u_n) - f_{\delta_0}(x,u_0) 
			- \sum_{i=1}^{\ell} f_{\delta_0}(x, \omega_i(x-y_{i,n}) ) \right| 
			| U_n | \rd x 
			\leq C_2 \e.
		\end{aligned}
	\]
Since $\e>0$ is arbitrary, we deduce that 
	\begin{equation}\label{eq:56}
		\limsup_{M \to \infty} 
		\sup_{n \geq 1}\int_{[|U_n| \geq M]} \left| f_{\delta_0}(x,u_n) - f_{\delta_0}(x,u_0) 
		- \sum_{i=1}^{\ell} f_{\delta_0}(x, \omega_i(x-y_{i,n}) ) \right| 
		|U_n| \rd x = 0.
	\end{equation}

	On the other hand, denote by $\chi_n^M(x) := \chi_{[|U_n| \leq M]}(x)$ 
the characteristic function of $[|U_n| \leq M]$. 
Since $u_n \to u_0$ and $u_n(x+y_{n,i}) \to \omega_i$ 
strongly in $L^p_{\rm loc}(\RN)$ for every $p < 2_\alpha^\ast$, 
for all $R>0$, Strauss' lemma, (G2), (G3), (G5) and the facts $|y_{n,i}| \to \infty$ 
and $|y_{n,i} - y_{n,j}| \to \infty$ for $i \neq j$ yield 
	\begin{equation}\label{eq:57}
		\begin{aligned}
			&\int_{B_R(0) } \chi_n^M(x) 
			\left\{ | f_{\delta_0}(x,u_n) - f_{\delta_0}(x,u_0)| 
			+ \sum_{i=1}^\ell | f_{\delta_0} (x, \omega_i( x -y_{n,i} )) | \right\}  |U_n| \rd x 
			\\
			\leq \, &  M \int_{B_R(0)} |f_{\delta_0} (x,u_n) - f_{\delta_0}(x,u_0)| 
			+ \sum_{i=1}^\ell | f_{\delta_0} (x, \omega_i( x -y_{n,i} )) |  \rd x \to 0 
			\quad {\rm as} \ n \to \infty,
			\\
			&\int_{B_R(y_{n,i})  } \chi_n^M(x) 
			\left\{  |f_{\delta_0}(x,u_0)| + \sum_{j \neq i}
			|f_{\delta_0} (x, \omega_j(x - y_{n,j} ) ) | \right\} |U_n| \rd x 
			\\
			& \quad + 
			\int_{ B_R(y_{n,i}) } \chi_n^M (x)
			 | f_{\delta_0}(x,u_n) - f_{\delta_0}(x,\omega_i(x-y_{n,i})) | |U_n| \rd x 
			\\
			\leq \,& M
			 \int_{B_R(0)} \Big\{ |f_{\delta_0} (x+y_{n,i},u_0(x+y_{n,i})| 
			 + \sum_{j \neq i} | f_{\delta_0} (x+y_{n,i}, \omega_j(x+y_{n,i} - y_{n,j} ) ) | 
			 \\
			 & \qquad \qquad 
			 + | f_{\delta_0}(x+y_{n,i},u_n(x+y_{n,i})) 
			- f_{\delta_0}(x+y_{n,i}, \omega_i(x))| \Big\}  \rd x 
			  \to 0 \quad {\rm as} \ n \to \infty.
		\end{aligned}
	\end{equation}
Writing $V_R := \RN \setminus ( B_R(0) \cup \bigcup_{i=1}^\ell B_R(y_{n,i} ))$ 
and recalling $|f_{\delta_0}(x,s)| \leq c_1(|s| + |s|^{2_\alpha^\ast})$, we have 
	\begin{equation}\label{eq:58}
		\begin{aligned}
			\int_{V_R} \chi_n^M |f_{\delta_0}(x,u_0) U_n| \rd x 
			&\leq c_1 \int_{V_R}  ( |u_0| + |u_0|^{2_\alpha^\ast-1}) |U_n| \rd x 
			\\
			&\leq c_1 \left( \| u_0 \|_{L^2( V_R )} \| U_n \|_{L^2(V_R)} 
			+ \| u_0 \|_{L^{2_\alpha^\ast} (V_R) }^{2_\alpha^\ast - 1} \| U_n \|_{L^{2_\alpha^\ast}} \right) 
			= o_R(1)
		\end{aligned}
	\end{equation}
where $o_R(1) \to 0$ as $R \to 0$ uniformly in $n$ and $M \geq 1$. 
Similarly, we also obtain 
	\begin{equation}\label{eq:59}
		\int_{V_R}  \chi_n^M |f_{\delta_0}(x,\omega_i(x-y_{n,i}))||U_n| \rd x 
		+ \int_{V_R } \chi_n^M |f_{\delta_0}(x,u_n)| \left( |u_0| 
		+ \sum_{i=1}^\ell |\omega_i(x-y_{n,i})| \right) 
		\rd x = o_R(1).
	\end{equation}

		Finally, noting that $f_{\delta_0}(x,s) s \leq 0$ for all $|s| \leq s_1$ and 
that $|f_{\delta_0}(x,s) s| \leq \e |s|^{2_\alpha^\ast}+ c_\e |s|^{p_0}$ 
for all $x \in \RN$ and $|s| \geq s_1$ where $p_0 \in (2,2_\alpha^\ast )$, 
by \eqref{eq:53}, we have 
	\[
		\begin{aligned}
			\int_{V_R} f_{\delta_0}(x,u_n) \chi_n^M u_n \rd x 
			&= \int_{V_R} f_{\delta_0}(x,\chi_n^M u_n) \chi_n^M u_n \rd x 
			\leq \int_{V_R \cap [|\chi_n^M u_n| \geq s_1]}
			f_{\delta_0} (x,\chi_n^M u_n) \chi_n^M u_n \rd x
			\\
			& \leq \e \| \chi_n^M u_n \|_{L^{2_\alpha^\ast}(V_R)}^{2_\alpha^\ast} 
			+ c_\e \| \chi_n^M u_n \|_{L^{p_0}(V_R)}^{p_0}.
		\end{aligned}
	\]
Recalling \eqref{eq:49}, $2 < p_0 < 2_\alpha^\ast$ and the definition of $V_R$, 
we obtain 
	\[
		\begin{aligned}
			& \limsup_{n\to \infty} 
			\| \chi_n^M u_n \|_{L^{p_0}(V_R)} 
			\\
			\leq \, & \limsup_{n\to \infty} 
			\left(  \left\| u_n - u_0 
			- \sum_{i=1}^\ell \omega_i (\cdot - y_{n,i}) \right\|_{L^{p_0}(V_R)} 
			+ \| u_0 \|_{L^{p_0}(V_R)} + \sum_{i=1}^{\ell} 
			\| \omega_i (\cdot - y_{n,i}) \|_{L^{p_0}(V_R)} \right)
			\\
			= \, & o_R(1),
		\end{aligned}
	\]
which implies $\limsup_{n \to \infty} 
| \int_{V_R} f_{\delta_0}(x,u_n) \chi_n^M u_n \rd x | \leq c \e + c_\e o_R(1)$ 
for some $c>0$. Thus, 
	\begin{equation}\label{eq:60}
		\limsup_{R \to \infty} \limsup_{n\to\infty} 
		\left| \int_{V_R} f_{\delta_0}(x,u_n) \chi_n^M u_n \rd x   \right| 
		\leq c \e.
	\end{equation}
Collecting \eqref{eq:58}, \eqref{eq:59} and \eqref{eq:60} with 
$U_n = u_n - u_0 - \sum_{i=1}^\ell \omega_i(x-y_{n,i})$, we observe that 
	\begin{equation}\label{eq:61}
		\limsup_{R \to \infty} \limsup_{n\to \infty} 
		\left| \int_{V_R}  \left\{ f_{\delta_0}(x,u_n) - f_{\delta_0}(x,u_0) - \sum_{i=1}^\ell 
					f_{\delta_0}(x, \omega_i(x-y_{n,i}))   \right\} \chi_n^M U_n \rd x \right| 
					\leq c \e. 
	\end{equation}
Since $\e > 0$ is arbitrary, by \eqref{eq:57} and \eqref{eq:61}, we obtain 
	\[
		\limsup_{n\to \infty} \left|\int_{\RN} 
		 \left\{ f_{\delta_0}(x,u_n) - f_{\delta_0}(x,u_0) - \sum_{i=1}^\ell 
					f_{\delta_0}(x, \omega_i(x-y_{n,i}))   \right\} \chi_n^M U_n \rd x 
					\right| = 0.
	\]
Combining this with \eqref{eq:56}, we observe that 
	\[
		\limsup_{n \to \infty}|I_n| = 0.
	\]

		In a similar way, we can also prove that 
$\limsup_{n \to \infty} |II_n| = 0$. Hence, by \eqref{eq:54}, 
we get $\| U_n \| \to 0$ as $n \to \infty$ and 
this completes the proof. 
	\end{proof}

\appendix


\section{Some technical Lemmas}
\label{section:A}


Here we prove some technical results. 
First, we show the following:

\begin{proposition}\label{definition:A.1}
\emph{(i)}
Let $0 < \delta_0 < 1$ and define $m(\xi)$ by 
	\[
		m(\xi) := \frac{1}{(1+4\pi^2 |\xi|^2)^\alpha - (1-\delta_0)}.
	\]
Set $K(x) := \scrF^{-1} m$. Then $K(x) \in C^\infty(\RN\setminus \{0\})$ and 
for any $k>0$ there exists a $c_k>0$ such that 
$|K(x)| \leq c_k( \chi_{B_1(0)}(x) |x|^{-N +2 \alpha} + \chi_{B_1(0)^c}(x) |x|^{-k} )$ 
for all $x \in \RN$.

\emph{(ii)} 
For any $g \in L^2(\RN)$, the equation 
	\begin{equation}\label{eq:62}
		(1-\Delta )^\alpha v - (1-\delta_0) v = g(x) \quad {\rm in}\ \RN
	\end{equation}
has a unique solution $v \in \Ha$. 
Moreover, when $g \in L^2(\RN) \cap L^\infty(\RN)$ and $g (x) \geq 0, \not\equiv 0$, 
then $v \in C^{\beta}_{\rm b}(\RN)$ for any $0<\beta < 2\alpha$ and 
$ v > 0$ in $\RN$. 
In addition, if ${\rm supp}\, g$ is compact and $g \in L^\infty(\RN)$ 
with $g(x) \geq 0$, 
then for any $k \in \RN$ one finds a $c_k>0$ such that 
$v(x) \leq c_k(1+|x|)^{-k}$ for all $x \in \RN$.

\emph{(iii)} 
Suppose that $g_i \in L^2(\RN) \cap L^\infty(\RN)$ 
$(i=1,2)$ satisfy 
$g_1 \leq g_2$. 
Let $v_i \in \Ha$ be a unique solution of \eqref{eq:62} 
with $g(x) = g_i(x)$. 
Then $v_1(x) \leq v_2(x)$ for $x \in \RN$. 
\end{proposition}

	\begin{proof}
(i) The smoothness of $K$ and the inequality $K(x) \leq C_0 |x|^{-N+2\alpha}$ 
for $|x| \leq 1$ follow from the arguments in \cite[$\S$4.4 of Chapter VI]{St-2}. 
For the decay estimate at infinity, since $(\Delta_\xi)^k m(\xi) \in L^1(\RN)$ 
provided $k > N/2$, we have 
$(4\pi^2|x|^2)^k K(x) = \scrF^{-1} \left( \Delta^k m \right) \in L^\infty (\RN)$. 
From this, the desired estimate follows.

	(ii) First, notice that a norm defined by 
	\[
		\| u \|^2 := (u,u), \quad 
		(u,v) := \la u, v \ra_\alpha - (1-\delta_0) \la u, v \ra_{L^2} 
		\quad {\rm for}\ u, v \in \Ha
	\]
is equivalent to $\| \cdot \|_\alpha$ since $0 < \delta_0 < 1$. 
Hence, \eqref{eq:62} has a unique solution $v \in \Ha$ for any 
$g \in L^2(\RN)$ due to the Lax-Milgram Theorem and 
it is expressed as $v = K \ast g$. 
Since \eqref{eq:62} is rewritten as 
$(1-\Delta)^\alpha v = (1-\delta_0) v + g$, 
we obtain 
	\[
		v = G_{2\alpha} \ast \left( (1-\delta_0) v + g \right). 
	\]
Thus if $ g \in L^\infty(\RN) \cap L^2(\RN)$, 
using the bootstrap argument and $\scrL^p_{2\alpha} \subset W^{2\alpha,p}(\RN)$, 
one can check $v \in C^{2\beta}_{\rm b}(\RN)$ for all $\beta \in (0,2\alpha)$.

	Let us assume $g(x) \geq 0, \not \equiv 0$ and $g \in L^\infty(\RN)$. 
Since $v= G_{2\alpha} \ast (( 1 - \delta_0 ) v + g  )$, $G_{2\alpha} > 0$ 
and $(G_{2\alpha} \ast g) >0$, we observe that 
	\[
			v(x) = \int_{\RN} G_{2\alpha} (y) 
			\left\{ (1-\delta_0) v(x-y) + g(x-y) \right\} \rd y 
			> (1-\delta_0) \int_{\RN} G_{2\alpha} (y) v(x-y) \rd y.
	\]
Noting $v \in C^\beta_{\rm b}(\RN)$ and $v(x) \to 0$ as $|x| \to \infty$, 
if $v(x_0) = \min_{\RN} v$ holds for some $x_0 \in \RN$, 
then it follows from $\| G_{2\alpha} \|_{L^1} = 1$ that 
	\[
		v(x_0) > (1-\delta_0) \int_{\RN} G_{2\alpha} (y) v(x_0-y) \rd y 
		\geq (1-\delta_0) \int_{\RN} G_{2\alpha} (y) v(x_0) \rd y 
		= (1-\delta_0) v(x_0).
	\]
By $0 < \delta_0 < 1$, we get $\min_{\RN} v = v(x_0) > 0$, however, this contradicts 
$v(x) \to 0$ as $|x| \to \infty$. Hence, $v$ does not have any global minimum 
on $\RN$ and this asserts $v(x) > 0$ for each $x \in \RN$.

		Finally, when $g(x)$ has the compact support, 
since $v = K \ast g$ and $K$ decays faster than any polynomial 
thanks to (i), it is easily seen that 
$v(x)$ also decays faster than any polynomial.

	(iii) Set $w := v_2 - v_1$. We observe that $w$ satisfies 
	\[
		(1-\Delta)^\alpha w - (1-\delta_0) w = g_2 - g_1 \geq 0 \quad {\rm in}\ \RN.
	\]
Using (ii), one has $w \geq 0$ in $\RN$ and 
$v_2 (x) \geq v_1(x)$ in $\RN$. 
	\end{proof}

	Next, in order to prove Proposition \ref{definition:3.5} and \eqref{eq:34}, 
we consider the extension problem observed in \cite{FF-15} (cf. \cite{CS-07}). 
For $X=(x,t) \in \ERN := \RN \times (0,\infty)$ and $u \in \Ha$, consider 
	\begin{equation}\label{eq:63}
		\left\{\begin{aligned}
			t^{1-2\alpha} (- \Delta_x + 1) w - (t^{1-2\alpha} w_t )_t &= 0 
			& &{\rm in}\ \ERN,\\
			w &= u & &{\rm on}\ \RN
		\end{aligned}\right.
	\end{equation}
where $\Delta_x = \sum_{i=1}^N \partial_{x_i}^2$. We set 
	\[
		X^\alpha := \left\{ w(x,t) :  \ERN \to \R \ |\ 
		\| w \|_{X^\alpha} < \infty  \right\}, 
		\quad 
		\| w \|_{X^\alpha}^2 := \int_{\ERN} 
		t^{1-2\alpha} \left( |\nabla w |^2 + w^2  \right) \rd X
	\]
where $\nabla = (\nabla_x , \partial_t)$. 
First we collect some facts. See, for instance, \cite{DD,FF-15}.

	\begin{proposition}\label{definition:A.2}
		\emph{(i)} There exists the trace operator $\Tr: X^\alpha \to \Ha$. 
		
		\emph{(ii)} For any $u \in \Ha$, \eqref{eq:63} has a unique solution 
		$w = Eu \in X^\alpha$. Furthermore, there exists a $\kappa_\alpha>0$ 
		such that $Eu$ satisfies 
			\[
				\int_{\ERN} t^{1-2\alpha} 
				\left( \nabla Eu \cdot \nabla \varphi + Eu \varphi  \right) 
				\rd X = \kappa_\alpha \la u,\Tr \varphi \ra_\alpha
			\]
		for all $u \in \Ha$ and $\varphi \in X^\alpha$. 
		
		\emph{(iii)} For every $u \in \Ha$ and $w \in X^\alpha$ with $\Tr w = u$, 
		one has 
			\[
				\kappa_\alpha \| u \|_{\alpha}^2 = 
				\| Eu \|_{X^\alpha}^2 \leq \| w \|_{X^\alpha}^2. 
			\]
		
		\emph{(iv)} If $u \in \Ha$ with $u \geq 0$, then $Eu \geq 0$ in $\ERN$. 
	\end{proposition}

Using these properties, we first show \eqref{eq:34}, namely, 

	\begin{lemma}\label{definition:A.3}
		For any $u \in \Ha$, $\| |u| \|_\alpha \leq \| u \|_\alpha$. 
		Moreover, the map $u \mapsto |u| : H^\alpha(\RN) \to H^\alpha(\RN)$ 
		is continuous. 
	\end{lemma}

	\begin{proof}
Let $u \in \Ha$. Then it is easily seen that 
$\| \, |Eu| \, \|_\alpha = \| Eu \|_\alpha < \infty$, hence, $|Eu| \in X^\alpha$. 
We can also check that $\Tr |Eu| = |u|$. 
Thus, by Proposition \ref{definition:A.2} (iii), we have 
$\kappa_\alpha \| |u| \|_\alpha^2 \leq \| \, |Eu| \,\|_{X^\alpha}^2 
= \| Eu \|_{X^\alpha}^2 = \kappa_\alpha \| u \|_\alpha^2$. 
For the continuity of the map $u \mapsto |u|$, 
let $u_n \to u_0$ in $\Ha$. 
From $Eu_n \to Eu_0$ in $X^\alpha$ due to Proposition \ref{definition:A.2} (iii), 
we have $|Eu_n| \to |Eu_0|$ in $X^\alpha$. 
By $\Tr |E u_n| = |u_n|$ and the boundedness of $\Tr$, 
we have $|u_n| = \Tr |E u_n| \to \Tr |E u_0| = |u_0|$ 
in $\Ha$. 
	\end{proof}

Now we prove Proposition \ref{definition:3.5}.

\begin{proof}[Proof of Proposition \ref{definition:3.5}]
The argument is similar to the case $\alpha = 1$ (see \cite{BK-79}). 
For $k \in \N$, set 
	\[
		a_k (x) := k \quad {\rm if} \ a(x) \geq k, \quad 
		:= a(x) \quad {\rm if} \ |a(x)| < k, \quad 
		:= - k \quad {\rm if} \ a(x) \leq -k.
	\]
Then thanks to \eqref{eq:36} and $A \in L^{N/(2\alpha)}(\RN)$, we have 
	\begin{equation}\label{eq:64}
		| a (x) - a_k(x) | \leq C_0 A(x) \quad 
		{\rm for\ each}\ (x,k) \in \RN \times [C_0,\infty) , \quad 
		\| a - a_k \|_{L^{N/(2\alpha)} (\RN) } \to 0.
	\end{equation}
The first step is to show: 

\medskip

\noindent
{\bf Step 1:}
{\sl For any $\e>0$ there exists a $\lambda_\e>0$ such that 
		\[
			\int_{\RN} |a| v^2 \rd x + \int_{\RN} |a_k| v^2 \rd x 
			\leq \e [ v ]^2_{H^\alpha } + 
			\lambda_\e \| v \|_{L^2}^2 
			\quad {\rm for\ all} \  v \in \Ha, \ k \geq 1
		\]
where $[u]^2_{H^\alpha} 
:= \int_{\RN \times \RN} |u(x)-u(y)|^2 / |x-y|^{N+2\alpha} \rd x \rd y$. 
}

\medskip

From \eqref{eq:36} and the definition of $a_k$, it follows that 
$|a(x)| + |a_k(x)| \leq C_0(1+A(x))$. 
Therefore, it suffices to prove 
	\begin{equation}\label{eq:65}
		\int_{\RN} A v^2 \rd x  \leq  \e [v]_{H^\alpha}^2 + 
		\lambda_\e \| v \|_{L^2}^2 .
	\end{equation}

We first notice that for $n \geq 1$, 
	\[
		\int_{\RN} A(x) v^2 \rd x 
		= \int_{[A<n]} A(x) v^2 \rd x + \int_{[A \geq n]} A(x) v^2 \rd x 
		\leq n \| v \|_{L^2}^2 + \int_{[A \geq n]} A(x) v^2 \rd x. 
	\]
Using H\"older's inequality and Sobolev's inequality 
for the second term, we obtain 
	\[
		\int_{[A \geq n]} A(x) v^2 \rd x 
		\leq  \| A \|_{L^{N/(2\alpha)} ([A \geq n]) } \| v \|_{L^{2_\alpha^\ast}}^2 
		\leq  C_S \| A \|_{L^{N/(2\alpha)} ([A \geq n]) } [v]_{H^\alpha}^2. 
	\]
Thus 
	\[
		\int_{\RN} A(x) v^2 \rd x  
		\leq  n \| v \|_{L^2}^2 + C_S \| A \|_{L^{N/(2\alpha)} ([A \geq n]) } [v]_{H^\alpha}^2.
	\]
Since $A \in L^{N/(2\alpha)}(\RN)$, choosing $n$ large enough, 
we get \eqref{eq:65} and Step 1 holds.

\medskip

\noindent
{\bf Step 2:} 
{\sl The operators $(1-\Delta)^\alpha - a(x) + \lambda_\e$ 
and $(1-\Delta)^\alpha - a_k(x) + \lambda_\e$ are 
coercive on $\Ha$ for all sufficiently small $\e>0$. 
}

\medskip

By Step 1, for sufficiently small $\e>0$, one sees that 
	\begin{equation}\label{eq:66}
		\begin{aligned}
			& \int_{\RN} (1 + 4 \pi^2 |\xi|^2)^\alpha |\wh{v}(\xi)|^2 \rd \xi 
			- \int_{\RN} a(x) v^2 \rd x + \lambda_\e \| v \|_{L^2}^2 
			\\
			\geq &\, \int_{\RN} (1 + 4 \pi^2 |\xi|^2)^\alpha |\hat{v}(\xi)|^2 \rd \xi 
			- \e [ v ]_{H^\alpha}^2  
			\geq 
			\frac{1}{2} \int_{\RN} (1 + 4 \pi^2 |\xi|^2)^\alpha |\wh{v}(\xi)|^2 \rd \xi
		\end{aligned}
	\end{equation}
and 
	\begin{equation}\label{eq:67}
		\int_{\RN} (1 + 4 \pi^2 |\xi|^2 )^\alpha |\wh{v}(\xi)|^2 \rd \xi 
		- \int_{\RN} a_k(x) v^2 \rd x + \lambda_\e \| v \|_{L^2}^2 
		\geq \frac{1}{2} \int_{\RN} (4 \pi^2 |\xi|^2 + m^2)^\alpha |\wh{v}(\xi)|^2 \rd \xi. 
	\end{equation}
Hence, Step 2 holds.

\medskip

Rewrite \eqref{eq:35} as follows:
	\begin{equation}\label{eq:68}
		(1-\Delta)^\alpha u - a(x) u +  \lambda_\e u =  \lambda_\e u
		\quad {\rm in}\ \RN.
	\end{equation}
Noting Step 2, we may find a unique solution $\psi_k \in H^\alpha(\RN)$ of
	\[
		(1 -\Delta )^\alpha \psi_k - a_k(x) \psi_k + \lambda_\e \psi_k = \lambda_\e u
		\quad {\rm in}\ \RN
	\]
for sufficiently small $\e>0$. 
From \eqref{eq:67}, one observes that $(\psi_k)$ is bounded in $H^\alpha(\RN)$. 
Furthermore, since $u$ is a unique solution of \eqref{eq:68} thanks to \eqref{eq:66}, 
we also see from \eqref{eq:64} that 
	\begin{equation}\label{eq:69}
		\psi_k \to u \quad {\rm strongly \ in} \ H^\alpha(\RN). 
	\end{equation}

	Next, let $w_{k} \in X^\alpha$ be a unique solution of \eqref{eq:63} with $u=\psi_k$. 
For $n \in \N$, set 
	\[
		\psi_{k,n}(x) := 
		\left\{\begin{aligned}
			& n & &{\rm if} \ \psi_k(x) \geq n,\\
			& \psi_k(x) & &{\rm if} \ |\psi_k(x)| < n,\\
			& -n & &{\rm if} \ \psi_k(x) \leq -n,
		\end{aligned}
		\right.
		\quad 
		w_{k,n}(X) := 
		\left\{\begin{aligned}
		& n & &{\rm if} \ w_k(X) \geq n,\\
		& w_k(X) & &{\rm if} \ |w_k(X)| < n,\\
		& -n & &{\rm if} \ w_k(X) \leq -n,
		\end{aligned}
		\right.
	\]
Remark that for every $p \geq 2$, 
	\[
		\begin{aligned}
			&|w_{k,n}|^{p-2} w_{k,n} \in X^\alpha \cap L^\infty(\ERN), 
			\quad 
			|\psi_{k,n}|^{p-2} \psi_{k,n} \in \Ha \cap L^\infty(\RN), 
			\\
			&
			\Tr ( |w_{k,n}|^{p-2} w_{k,n} ) = |\psi_{k,n}|^{p-2} \psi_{k,n}, \quad 
			\| w_{k,n} - w_k \|_{X^\alpha} \to 0, \quad 
			\| \psi_{k,n} - \psi_k \|_{\alpha}  \to 0. 
		\end{aligned}
	\]

\medskip

\noindent
{\bf Step 3:}
{\sl Assume that $u, \psi_k \in L^p(\RN)$ and 
$\psi_k \to u$ strongly in $L^p(\RN)$ for some $p > 2$. 
Then  
	\[
		|u|^{p/2} \in H^\alpha(\RN), \quad 
		\kappa_\alpha \| |u|^{p/2} \|_{H^\alpha}^2 
		\leq \| |w|^{p/2} \|_{X^\alpha}^2 \leq C_1 \| u \|_{L^{p}}^{p} 
		\quad \text{\sl where} \ w = Eu \in X^\alpha.
	\]
}

\medskip

	We use $|w_{k,n}|^{p-2} w_{k,n}$ as a test function to \eqref{eq:63} 
with $w=w_k$ to get 
	\[
		\begin{aligned}
		&\int_{\ERN} t^{1-2\alpha} 
		\left\{ \nabla w_k \cdot \nabla (|w_{k,n}|^{p-2} w_{k,n}) 
		+  w_k |w_{k,n}|^{p-2} w_{k,n} \right\} 
		\rd X 
		\\
		= \, & \kappa_\alpha \int_{\RN} 
		\left( a_k \psi_k - \lambda_\e \psi_k + \lambda_\e u \right) 
		|\psi_{k,n}|^{p-2} \psi_{k,n} \rd x. 
		\end{aligned}
	\]
Notice that 
	\[
		\nabla w_k \cdot \nabla (|w_{k,n}|^{p-2} w_{k,n} ) 
		= (p-1) |w_{k,n}|^{p-2} | \nabla w_{k,n} |^2 
		= \frac{4}{p^2}(p-1) | \nabla (|w_{k,n}|^{p/2}) |^2. 
	\]
Furthermore, by $|w_{k,n}| \leq |w_k|$, 
$|\psi_{k,n}| \leq |\psi_k|$, 
$w_k w_{k,n} = |w_k||w_{k,n}|$, 
$\psi_{k,n} \psi_k = |\psi_{k,n}||\psi_k|$, $|a_k(x)| \leq k$ 
and $|\psi_{k,n}|^{p/2} \in \Ha$, 
it follows from Step 1 that 
	\[
		\begin{aligned}
			\int_{\ERN} t^{1-2\alpha} w_k |w_{k,n}|^{p-2} w_{k,n} \rd X  
			&\geq \int_{\ERN} t^{1-2\alpha} (|w_{k,n}|^{p/2})^2 \rd X, 
			\\ 
			\int_{\RN} \lambda_\e \psi_k |\psi_{k,n}|^{p-2} \psi_{k,n} \rd x 
			 &\geq \int_{\RN} \lambda_\e |\psi_{k,n}|^{p} \rd x ,
			\\
			\int_{\RN} a_k \psi_k |\psi_{k,n}|^{p-2} \psi_{k,n}  \rd x
			&= \left(\int_{[|\psi_k| < n]} + \int_{[|\psi_{k}| \geq n]}\right) 
			a_k \psi_k |\psi_{k,n}|^{p-2} \psi_{k,n}  \rd x 
			\\
			& \leq \int_{\RN} |a_k| (|\psi_{k,n}|^{p/2})^2 \rd x 
			+ n^{p-1} k \int_{[|\psi_{k}| \geq n]} |\psi_k| \rd x
			\\
			& \leq \e [|\psi_{k,n}|^{p/2}]_{H^\alpha}^2 
			+ \lambda_\e \int_{\RN} |\psi_{k,n}|^p \rd x 
			+ k n^{p-1} \int_{[|\psi_{k}| \geq n]} |\psi_k| \rd x.
		\end{aligned}
	\]
Therefore, we obtain 
	\[
		\begin{aligned}
			& \int_{\ERN} t^{1-2\alpha} 
			\left\{ \frac{4}{p^2}(p-1) | \nabla (|w_{k,n}|^{p/2} ) |^2 
			+ (|w_{k,n}|^{p/2})^2 \right\} 
			\rd X  
			\\
			\leq & \,  \kappa_\alpha 
			\left[ 
			\e [|\psi_{k,n}|^{p/2}]_{H^\alpha}^2  
			+ k n^{p-1} \int_{[|\psi_k| \geq n]} |\psi_k| \rd x
			+ \lambda_\e \int_{\RN} |u| |\psi_{k,n}|^{p-1} \rd x
			\right] . 
		\end{aligned}
	\]
Since Proposition \ref{definition:A.2} asserts 
	\[
		\kappa_\alpha \| |\psi_{k,n}|^{p/2} \|_{H^\alpha}^2 
		\leq \| |w_{k,n}|^{p/2} \|_{X^\alpha}^2, 
		\quad n^{p-1} \leq |\psi_k|^{p-1} \quad {\rm on}\ 
		[|\psi_k| \geq n],
	\]
choosing 
	\[
		\e = \frac{1}{2} \frac{4(p-1)}{p^2}, 
	\]
we finally obtain 
	\begin{equation}\label{eq:70}
		\| |w_{k,n}|^{p/2} \|_{X^\alpha}^2 
		\leq C_{p,\alpha} 
		\left[  k \int_{[|\psi_k| \geq n]} |\psi_k|^p \rd x
		+ \int_{\RN} |u| |\psi_{k,n}|^{p-1} \rd x   \right].
	\end{equation}

	Now let us consider the case where $u,\psi_k \in L^p(\RN)$ and 
$\psi_k \to u$ strongly in $L^p(\RN)$. 
By H\"older's inequality and the definition of $\psi_{k,n}$, we have 
	\[
		\int_{\RN} |u| |\psi_{k,n}|^{p-1} \rd x 
		\leq \| u \|_{L^p} \| \psi_{k} \|_{L^{p}}^{p-1}
	\]
and the right hand side in \eqref{eq:70} is bounded as $n \to \infty$. 
Since $w_{k,n} \to w_k$ strongly in $X^\alpha$, we observe that 
$|w_{k,n}|^{p/2} \rightharpoonup |w_k|^{p/2}$ weakly in $X^\alpha$. 
Letting $n \to \infty$ in \eqref{eq:70}, one has
	\begin{equation}\label{eq:71}
		\| |w_{k}|^{p/ 2} \|_{X^\alpha}^2 
		\leq C_1 \| u \|_{L^{p}} \| \psi_k \|_{L^{p}}^{p-1}.
	\end{equation}
Since $\Tr (|w_k|^{p/2}) = |\psi_k|^{p/2}$, Proposition \ref{definition:A.2} gives 
	\[
		\kappa_\alpha \| |\psi_k|^{p/2} \|_{H^\alpha}^2 
		\leq \| |w_k|^{p/2} \|_{X^\alpha}^2 
		\leq C_1 \| u \|_{L^{p}} \| \psi_k \|_{L^{p}}^{p-1}.
	\]
Thus by Sobolev's inequality, we get 
	\begin{equation}\label{eq:72}
		\| |\psi_k|^{p/2} \|_{L^{2_\alpha^\ast}}^2 \leq C_0 \| |\psi_k|^{p/2} \|_{H^\alpha}^2 
		\leq C \| u \|_{L^{p}} \| \psi_k \|_{L^{p}}^{p-1}. 
	\end{equation}
Recalling $\psi_k \to u$ strongly in $L^p(\RN)$ and 
letting $k \to \infty$ in \eqref{eq:71},  
we observe that 
$(|w_{k}|^{p/2})$ is bounded in $X^\alpha$, 
$|w_{k}|^{p/2} \rightharpoonup |w|^{p/2}$ weakly in $X^\alpha$ and 
	\[
		\| |w|^{p/2} \|_{X^\alpha}^2 \leq C_1 \| u \|_{L^{p}}^{p}
	\]
where $w = Eu \in X^\alpha$. 
Noting ${\rm Tr}\, (|w|^{p/2}) = |u|^{p/2}$, we have 
	\[
		|u|^{p/2} \in H^\alpha(\RN), \quad 
		\kappa_\alpha \| |u|^{p/2} \|_{H^\alpha}^2 
		\leq \| |w|^{p/2} \|_{X^\alpha}^2 \leq C_1 \| u \|_{L^{p}}^{p}
	\]
and Step 3 holds.

\medskip

\noindent
{\bf Step 4:} {\sl Conclusion}

\medskip

By Step 3 and \eqref{eq:72}, if $u,\psi_k \in L^p(\RN)$ and $\psi_k \to u$ strongly in 
$L^p(\RN)$, then 
	\begin{equation}\label{eq:73}
		\| |u|^{p/2} \|_{L^{2^\ast_\alpha}}^2 
		\leq C \| |u|^{p/2} \|_{H^\alpha}^2 
		\leq C \kappa_\alpha^{-1} \| |w|^{p/2} \|_{X^\alpha}^2 
		\leq C \| u \|_{L^p}^p, \quad 
		\| |\psi_{k}|^{p/2} \|_{L^{2_\alpha^\ast}}^2 
		\leq C \| u \|_{L^p}\| \psi_k \|_{L^{p}}^{p-1}. 
	\end{equation}
Now we select $p = p_1 := 2_\alpha^\ast > 2$. 
From \eqref{eq:69}, the assumptions of Step 3 are satisfied and 
	\[
		\| |u|^{p_1/2} \|_{L^{2^\ast_\alpha}}^2 
		\leq C \| |u|^{p_1/2} \|_{H^\alpha}^2 
		\leq C \kappa_\alpha^{-1} \| |w|^{p_1/2} \|_{X^\alpha}^2 
		\leq C \| u \|_{L^{p_1}}^{p_1}, \quad 
		\| |\psi_{k}|^{p_1/2} \|_{L^{2_\alpha^\ast}}^2 
		\leq C \| u \|_{L^{p_1}}\| \psi_k \|_{L^{p_1}}^{p_1-1}. 
	\]
From this, one observes that the assumptions of Step 3 holds 
for any $2 \leq p < p_1 2^\ast_\alpha / 2$. Hence, setting 
$p_2 := p_1 2^\ast_\alpha /2$, \eqref{eq:73} holds for each $p < p_2$. 
Again, the assumptions of Step 3 hold for each $p < p_3 := p_2 2^\ast_\alpha / 2$. 
Repeating this argument and noting $2^\ast_\alpha /2 > 1$, we observe that 
\eqref{eq:73} holds for any $p < \infty$, which implies 
$u \in L^p(\RN)$ for any $2 \leq p < \infty$. 
This completes the proof. 
\end{proof}

\subsection*{Acknowledgement}
The author would like to thank Professor Tatsuki Kawakami, 
Professor Tohru Ozawa and Professor Kazunaga Tanaka 
for valuable comments and fruitful discussions on the topic of 
this paper.

\end{document}